\documentclass[11pt,a4paper]{amsart}

\usepackage{setspace}
\usepackage{amsthm}
\usepackage{amssymb,amsmath,bbm}
\usepackage{amsfonts}
\usepackage[dvips]{graphicx}
\usepackage{pstricks}
\usepackage[ansinew]{inputenc}
\usepackage{enumerate}
\usepackage{multirow}
\usepackage{subcaption}

\newhsbcolor{dar}{1 0.9 0.55}

\newtheorem{theorem}{Theorem}[section]

\newtheorem{lemma}[theorem]{Lemma}
\newtheorem{proposition}[theorem]{Proposition}
\newtheorem{problem}[theorem]{Problem}

\theoremstyle{definition}

\newtheorem{assumption}[theorem]{Assumption}

\theoremstyle{remark}
\newtheorem{remark}[theorem]{Remark}

\newcommand{\R}{\mathbb{R}}

\newcommand{\D}{\displaystyle}

\newcommand{\Exp}{\operatornamewithlimits{Exp}}
\newcommand{\E}{\mathbb{E}}
\newcommand{\p}{\mathbb{P}}

\newcommand{\argmin}{\operatornamewithlimits{argmin}}
\newcommand{\argmax}{\operatornamewithlimits{argmax}}

\begin{document}
\title[Expected Supremum Representation and Optimal Stopping]{Expected Supremum Representation and Optimal Stopping}
\author{Luis H. R. Alvarez E.}\thanks{Turku School of Economics, Department of Accounting and Finance, 20014 University of Turku, Finland,
e-mail:luis.alvarez@tse.fi}
\author{Pekka Matom\"aki}\thanks{Turku School of Economics, Department of Accounting and Finance, 20014 University of Turku, Finland,
e-mail:pjsila@utu.fi}
\date{May 7, 2015}

\begin{abstract}
We consider the representation of the value of an optimal stopping problem of a linear diffusion as an expected supremum of a known function.
We establish an explicit integral representation of this function by utilizing the explicitly known joint probability distribution of the extremal processes.
We also delineate circumstances under which the value of a stopping problem induces directly this representation and show how it is connected with the monotonicity of the generator.
We compare our findings with existing literature and show, for example, how our representation is linked to the smooth fit principle and how it coincides
with the optimal stopping signal representation.
The intricacies of the developed integral representation are explicitly illustrated in various examples arising in financial applications of optimal stopping.
\end{abstract}
\maketitle
\thispagestyle{empty}
\clearpage \setcounter{page}{1}
\section{Introduction}

It is well-known from the literature on stochastic processes that the probability distributions of first hitting times are closely related to the probability
distributions of the running supremum and running infimum of the underlying diffusion. Consequently, the question of whether a linear diffusion has exited
from an open interval prior to a given date or not can be answered by studying the behavior of the extremal processes up to the date in question. If the extremal
processes have remained in the open interval up to the particular date, then the process has not yet hit the boundaries and vice versa. In
this study we utilize this connection and develop a representation of the value function of an optimal stopping problem as the expected supremum of a function
with known properties in the spirit of the pioneering work by \cite{FoKn1,FoKn2} and the subsequent extension to optimal stopping by \cite{ChSaTa}.

The relatively recent literature on stochastic control theory indicates that the connection between, among others, the value functions and extremal processes in optimal stopping and
singular stochastic control problems goes far beyond the standard connection between first hitting times and the running supremum
and infimum of the underlying process (see, for example, \cite{BaBa,BaElKa,BaFo,BaRi,ElKaMe,ElKaFo}). Essentially, in these studies the determination of the
optimal policy and its value is shown to be equivalent with the existence of an appropriate optional projection involving the running supremum of a progressively
measurable process. The advantage of the representation utilized in these studies is that it is very general and applies
also outside the standard Markovian and infinite horizon setting.
Moreover, it can be utilized for studying and solving other than just optimal stopping and singular stochastic control problems as well. For example, as was shown in \cite{BaElKa,BaFo},
the approach is applicable in the analysis of the Gittins-index familiar from the literature on multi-armed bandits (cf. \cite{ElKarKar94,Git79,GitGla77,GitJon79,Kar1984}).

Instead of establishing directly how the value of an optimal stopping problem can be expressed as an expected supremum, we take an alternative route and compute
first the joint probability distribution of the running supremum and running infimum of the underlying diffusion at an independent exponentially distributed random
time. We then compute explicitly the expected value of the supremum of an unknown function subject to a set of
monotonicity and regularity conditions. Setting this expected value equal with the value of an optimal stopping problem then results into a functional identity
from which the unknown function can be explicitly determined. In the single boundary setting the function admits a relatively simple characterization in terms of
the minimal excessive mappings for the underlying diffusion (cf. \cite{BaBa}). We find that the required monotonicity of the function needed for the representation is closely
related with the monotonicity
of the generator on the state space of the underlying diffusion. However, since only the sign of the generator typically affects the determination of the optimal strategy
and its value, our results
demonstrate that not all single boundary problems can be represented as the expected supremum of a monotonic function. We also investigate the regularity properties of the
function needed for the representation and show that
it needs not be continuous. More precisely, we find that if the optimal boundary is attained at a point where the exercise payoff is not differentiable, then the
function needed for the representation is only upper semicontinuous. This is a result which is in line with the findings by \cite{ChSaTa}.

In the two boundary setting the representation becomes more involved and takes an integral form where the
integration bounds are interdependent due to the dependence of the two extremal processes. However, since
the representation is based on the minimal $r$-excessive functions and the scale of the underlying diffusion, our approach results into a
representation which can be efficiently utilized in numerical computations. We also compare our representation with previous representations.
Given that our approach is based on the study by \cite{ChSaTa} it naturally coincides with their representation the main difference being that we compute the expected
supremum explicitly and in that way state an explicit representation of the unknown function needed for the representation. We also establish that our representation
coincides with the stopping signal representation originally developed in \cite{BaBa}. Hence, our findings provide an explicit connection between these two seemingly different approaches.
Furthermore, we also demonstrate that the continuity requirement of the functional form needed for the representation is equivalent with the standard smooth fit principle. In this way, our
study provides a link between the usual (e.g. free boundary/variational inequalities) approach and the more recent approaches based on the running supremum.
In line with our findings in the single boundary case, our results indicate that the function needed for the representation does not need to be continuous. In this way,
our numerical results appear to show that
the stopping signal representation developed in \cite{BaBa} applies also in a nonsmooth environment.

The contents of this study is as follows. In section two we formulate the considered problem, characterize the underlying stochastic dynamics, and state a set of
auxiliary results needed in the subsequent analysis of the problem. Section three focuses on a single boundary setting where the optimal rule is to exercise as soon
as a given exercise threshold is exceeded.  The general two-boundary
case is then investigated in detail in section four. Finally, section five concludes our study.

\section{Problem Formulation}

\subsection{Underlying stochastic dynamics}
We consider a linear, time homogeneous and regular diffusion process $X=\{X(t);t\in[0,\xi)\}$, where $\xi$ denotes the possible infinite life time of the
diffusion. We assume that the diffusion is defined on a complete filtered probability space $(\Omega,\p,\{\mathcal{F}_t\}_{t\geq0},\mathcal{F})$, and that the state
space of the diffusion is $\mathcal{I}=(a,b)\subset\R$. Moreover, we assume that the diffusion does not die inside $\mathcal{I}$, implying that the boundaries $a$
and $b$ are either natural, entrance, exit or regular (see \cite{BorSal02}, pp. 18-20 for a characterization of the boundary behaviour of diffusions). If the boundary
is regular, we assume that the process is either killed or reflected at that boundary. Furthermore we will denote by $I_t=\inf_{0\leq s\leq t}X_s$ the running infimum
and by $M_t=\sup_{0\leq s\leq t}X_s$ the running supremum process of the considered diffusion $X_t$.

As usually, we denote by $\mathcal{A}$ the differential operator representing the infinitesimal generator of $X$. For a given smooth mapping $f:\mathcal{I}\mapsto\R$ this operator is given by
\begin{align*}
(\mathcal{A}f)(x)=\frac{1}{2}\sigma^2(x)\frac{d^2}{dx^2}f(x)+\mu(x)\frac{d}{dx}f(x),
\end{align*}
where $\mu:\mathcal{I}\mapsto\R$ and $\sigma:\mathcal{I}\mapsto\R_+$ are given continuous mappings. As is known from the classical
theory on linear diffusions,
there are two linearly independent {\it fundamental solutions}
$\psi(x)$ and $\varphi(x)$ satisfying a set of appropriate boundary
conditions based on the boundary behavior of the process $X$ and
spanning the set of solutions of the ordinary differential equation
$(\mathcal{G}_ru)(x)=0$, where $\mathcal{G}_r=\mathcal{A}-r$ denotes the differential operator associated with the diffusion $X$ killed at the constant rate $r$.
Moreover, $\psi'(x)\varphi(x) - \varphi'(x)\psi(x) = BS'(x),$
where $B>0$ denotes the constant Wronskian of the fundamental
solutions and
$$
S'(x)=\exp\left(-\int^x\frac{2\mu(t)}{\sigma^2(t)}dt\right)
$$
denotes the density of the scale function of $X$ (for a comprehensive characterization of the fundamental solutions, see \cite{BorSal02}, pp. 18-19). The
functions $\psi$ and $\varphi$ are minimal in the sense that any non-trivial $r$-excessive mapping for $X$ can be expressed as a combination of these two
(cf. \cite{BorSal02}, pp. 32--35). Given the fundamental solutions, let $u(x)=c_1\psi(x) + c_2\varphi(x), c_1,c_2\in \mathbb{R}$ be an arbitrary twice continuously
differentiable $r$-harmonic function and define for sufficiently smooth mappings $g:\mathcal{I}\mapsto\R$ the functional
\begin{align*}
(L_u g)(x) = g(x) \frac{u'(x)}{S'(x)}-\frac{g'(x)}{S'(x)}u(x) = c_1(L_\psi g)(x)+c_2(L_\varphi g)(x)
\end{align*}
associated with the representing measure for $r$-excessive functions (cf. \cite{Salminen1985}). Noticing that if $g$ is twice continuously differentiable, then
\begin{align}
(L_u g)'(x)=-(\mathcal{G}_rg)(x)u(x)m'(x)\label{linearityGenerator}
\end{align}
where
$
m'(x) = 2/(\sigma^{2}(x)S'(x))
$
denotes the density of the speed measure $m$ of $X$.
Hence, we find that
\begin{align}
(L_u g)(z)-(L_u g)(y)=\int_z^y(\mathcal{G}_rg)(t)u(t)m'(t)dt\label{canonical}
\end{align}
for any $a<z<y<b$. Especially, if $g$ is twice continuously differentiable, $\mathbbm{1}=\mathbbm{1}_{\mathcal{I}}(x)$, and $a<z<y<b$, then the (symmetric) function
\begin{align}
R(z,y)=\frac{(L_u g)(z)-(L_u g)(y)}{(L_u \mathbbm{1})(z)-(L_u \mathbbm{1})(y)}\label{ratio}
\end{align}
satisfies the limiting condition
\begin{align}
\lim_{z\uparrow y}R(z,y)=-\frac{1}{r}(\mathcal{G}_rg)(y)\label{limitingratio}
\end{align}
which is independent of the harmonic function $u$.
Finally, we denote by $\mathcal{L}_r^1(\mathcal{I})$ the class of measurable functions $f:\mathcal{I}\mapsto\R_+$ satisfying the integrability condition
$$
\mathbb{E}_x\int_0^\infty e^{-rs}|f(X_s)|ds<\infty
$$
for all $x\in \mathcal{I}$. As is known from the literature on linear diffusions, the expected cumulative present value of a continuous function $f\in \mathcal{L}_r^1(\mathcal{I})$, that is,
$$
(R_rf)(x)=\mathbb{E}_x\int_0^\infty e^{-rs}f(X_s)ds
$$
can be expressed as
\begin{align}
(R_rf)(x)= B^{-1}\varphi(x)\int_a^x \psi(y)f(y)m'(y)dy+B^{-1}\psi(x)\int_x^b \varphi(y)f(y)m'(y)dy.\label{Green}
\end{align}

\subsection{The Optimal Stopping Problem and Auxiliary Results}

In this paper our objective is to examine an optimal stopping problem
\begin{align}\label{eq prob}
V(x)=\sup_{\tau}\E_x\left[e^{-r\tau}g(X_\tau)\right]
\end{align}
for exercise payoff functions $g$ satisfying a set of sufficient regularity conditions and establish a representation of the value $V(x)$ as
the expected supremum of an appropriately chosen function along
the lines of the pioneering studies \cite{BaElKa},\cite{BaFo},\cite{ChSaTa}, \cite{ElKaFo}, \cite{ElKaMe}, \cite{FoKn1}, \cite{FoKn2}. Our main results are
based on the following two representation theorems originally established in \cite{ChSaTa}. The first theorem focuses on the case of a one-sided stopping boundary.

\begin{theorem}(\cite{ChSaTa}, Theorem 2.5)\label{theorem abo1}
Let $X_t$ be a Hunt process on $\mathcal{I}$ and $T\sim\Exp (r)\perp X_t$. Assume that the exercise payoff $g$ is non-negative, lower semicontinuous, and satisfies the condition
$\E_x\left[\sup_{t\geq 0}e^{-rt}g(X_t)\right]<\infty$ for all $x\in \mathcal{I}$.
Assume also that there exists an upper semicontinuous $\hat{f}$ and a point $y^{\ast}\in \mathcal{I}$ such that
\begin{enumerate}
\item[(a)] $\hat{f}(x)\leq0$ for $x< y^{\ast}$, $\hat{f}(x)$ is non-decreasing and positive for $x\geq y^{\ast}$,

\item[(b)] $\E_x\left[\sup_{0\leq t\leq T}\hat{f}(X_t)\right]=g(x)$ for $x\geq y^{\ast}$, and
\item[(c)] $\E_x\left[\sup_{0\leq t\leq T}\hat{f}(X_t)\right]\geq g(x)$ for $x\leq y^{\ast}$.
\end{enumerate}
Then
\begin{align}\label{aboarvo1}
V(x)=\E_x\left[\sup_{0\leq t\leq T}\hat{f}(X_t)\mathbbm{1}_{[y^{\ast},b)}(X_t)\right]=\E_x\left[\hat{f}(M_T)\mathbbm{1}_{[y^{\ast},b)}(M_T)\right]
\end{align}
and  $\tau^*=\inf\{t\geq0\mid X_t>y^{\ast}\}$ is an optimal stopping time.
\end{theorem}

This theorem essentially says that if we can find a function satisfying the required conditions (a)-(c), then the optimal stopping policy for \eqref{eq prob}
constitutes an one-sided threshold rule. Moreover, in that case we also notice that the value can be expressed
as an expected supremum attained at an independent exponential random time. As we will prove later in this paper, the reverse argument is  also sometimes true: under certain circumstances
based on the behavior of the infinitesimal generator of the underlying diffusion the value of
the optimal policy generates a continuous and monotone function $\hat{f}$ for which the representation \eqref{aboarvo1} is valid. However, as we will
point out in Example 1, all single boundary stopping problems cannot be represented as proposed in Theorem \ref{theorem abo1}.

The second representation  theorem established in \cite{ChSaTa} focusing on two-sided stopping rules is summarized in the following\footnote{Both Theorem \ref{theorem abo1} and
Theorem \ref{theorem abo2} are slightly modified versions of the original ones. Three minor misprints have been corrected based on a personal communication with P. Salminen}.
\begin{theorem}(\cite{ChSaTa}, Theorem 2.7)\label{theorem abo2}
Let $X_t$ be a Hunt process on $\mathcal{I}$ and $T\sim\Exp (r)\perp X_t$.  Assume that the exercise payoff $g$ is non-negative, lower semicontinuous, and satisfies the condition
$\E_x\left[\sup_{t\geq 0}e^{-rt}g(X_t)\right]<\infty$ for all $x\in \mathcal{I}$.
Assume also that there exists an upper semicontinuous $\hat{f}$ and a pair of points $(z^{\ast},y^{\ast})$ such that
\begin{enumerate}
\item[(a)] $\hat{f}(x)\leq 0$ for $x\in (z^\ast,y^\ast)$, $\hat{f}(x)$ is non-increasing on $(a,z^\ast)$, nondecreasing on $(y^\ast,b)$, and positive on $(a,z^\ast)\cup (y^\ast,b)$,
\item[(b)] $\E_x\left[\sup_{0\leq t\leq T}\hat{f}(X_t)\right]=g(x)$ for $x\notin[z^{\ast},y^{\ast}]$, and
\item[(c)] $\E_x\left[\sup_{0\leq t\leq T}\hat{f}(X_t)\right]\geq g(x)$ for $x\in[z^{\ast},y^{\ast}]$.
\end{enumerate}
Then
\[V(x)=\E_x\left[\sup_{0\leq t\leq T}\hat{f}(X_t)\mathbbm{1}_{(a,z^{\ast})\cup(y^{\ast},b)}(X_t)\right]
=\E_x\left[\left[\hat{f}(I_T)\mathbbm{1}_{(a,z^{\ast}]}(I_T)\right]\lor\left[\hat{f}(M_T)\mathbbm{1}_{[y^{\ast},b)}(M_T)\right] \right]\]
and $\tau^*=\inf\{t\geq0\mid X_t\notin[z^{\ast},y^{\ast}]\}$ is an optimal stopping time.
\end{theorem}

Theorem \ref{theorem abo2} essentially states a set of conditions extending the one sided representation considered in Theorem \ref{theorem abo1} to the two-sided setting.
It is, however, worth noticing that these theorems do not tell us how to come up with such functions $\hat{f}$. Our objective is to identify these functions in the ordinary linear diffusion
setting and in this way establish a link between the supremum representation and the standard solution techniques.

Before proceeding in our analysis, we first establish two auxiliary lemmata characterizing the joint probability distribution of the extreme
processes and the underlying diffusion at an independent exponentially distributed random time. Our first findings on the joint probability distribution of
$M_T$ and $I_T$ are summarized in the following.
\begin{lemma}\label{probabilitydist}
The joint probability distribution of the extreme processes $M_t$ and $I_t$ at an independent exponentially distributed random time $T$ reads for all $x\in (i,m)$ as
\begin{align}
\p_x(I_T\leq i,M_T\leq m)=-\frac{\psi(x)}{\psi(m)}+\frac{\hat{\varphi}_m(x)}{\hat{\varphi}_m(i)}+\frac{\hat{\psi}_i(x)}{\hat{\psi}_i(m)},\label{jointprob}
\end{align}
where $\hat{\varphi}_m(x)=\psi(m)\varphi(x)-\varphi(m)\psi(x)$ and $\hat{\psi}_i(x)=\varphi(i)\psi(x)-\psi(i)\varphi(x)$. The marginal distributions
read as
\begin{align}
\p_x(M_T\leq m)=1-\frac{\psi(x)}{\psi(m)}\label{eq Pmax}
\end{align}
for $x\in(a,m)$ and as
\begin{align*}
\p_x(I_T\leq i)=\frac{\varphi(x)}{\varphi(i)}
\end{align*}
for $x\in (i,b)$.
\end{lemma}
\begin{proof}
It is known that (see \cite{BorSal02}, pp. 25--26)
\begin{align*}
\p_x(M_T\leq m)=1-\p_x(\tau_m< T)=1-\E_x\left[e^{-r\tau_m}\right]=1-\frac{\psi(x)}{\psi(m)}
\end{align*}
for all $x\in(a,m)$. In a completely analogous fashion, we find that  for $x\in(i,b)$ it holds (see \cite{BorSal02}, pp. pp. 25--26)
\begin{align*}
\p_x(I_T\leq i)=\p_x(\tau_i< T)=\frac{\varphi(x)}{\varphi(i)}.
\end{align*}
For determining the joint probability distribution, we first notice that for all $x\in (i,m)$ we have
\begin{align*}
\p_x(I_T\geq i, M_T\leq m)&=1-\p_x(\tau_{i,m}\leq T)=1-\E_x\left[e^{-r\tau_{i,m}}\right]=1-\frac{\hat{\varphi}_m(x)}{\hat{\varphi}_m(i)}-\frac{\hat{\psi}_i(x)}{\hat{\psi}_i(m)}
\intertext{and that}
\p_x(I_T\leq i,M_T\leq m)&=\p_x(M_T\leq m)-\p_x(I_T\geq i,M_T\leq m)
=-\frac{\psi(x)}{\psi(m)}+\frac{\hat{\varphi}_m(x)}{\hat{\varphi}_m(i)}+\frac{\hat{\psi}_i(x)}{\hat{\psi}_i(m)},
\end{align*}
where $\p_x(M_T\leq m)$ was calculated already in \eqref{eq Pmax}.
\end{proof}

Let $\tilde{X}_t=\{X_t;t<\tau_v\}$, $\tau_v=\inf\{t\geq 0: X_t\geq v\}$, denote the diffusion $X$ killed at $v\in \mathcal{I}$ and
$\hat{X}_t=\{X_t;t<\tau_i\}$, $\tau_i=\inf\{t\geq 0: X_t\leq i\}$, denote the diffusion $X$ killed at $i\in \mathcal{I}$. Given these diffusions, we
define $\hat{M}_t=\sup\{\hat{X}_s,s\leq t\}$ and $\tilde{I}_t=\inf\{\tilde{X}_s,s\leq t\}$. We can now establish the following useful result needed later in the
characterization of the value of a stopping problem as an expected supremum in the two-boundary setting.
\begin{lemma}\label{jointprobabilitydist}
Assume that $a<i<v<b$. Then,
\begin{align*}
\mathbb{P}_x\left[\hat{X}_T\in dy|\hat{M}_T=v\right]&=\frac{r\hat{\psi}_i(y)m'(y)dy}{\frac{\hat{\psi}_i'(v)}{S'(v)}-B}\\
\mathbb{P}_x\left[\tilde{X}_T\in dy|\tilde{I}_T= i\right]&=\frac{r\hat{\varphi}_v(y)m'(y)dy}{-B-\frac{\hat{\varphi}_v'(i)}{S'(i)}}.
\end{align*}
for all $x\in (i,v)$. Consequently, if $h:(i,v)\mapsto \mathbb{R}$ is integrable, we have
\begin{align*}
\E_x[h(\hat{X}_T)|\hat{M}_T=v]&=\frac{\int_i^vh(y)\hat{\psi}_i(y)m'(y)dy}{\int_i^v\hat{\psi}_i(y)m'(y)dy}\\
\E_x[h(\tilde{X}_T)|\tilde{I}_T= i]&=\frac{\int_i^vh(y)\hat{\varphi}_v(y)m'(y)dy}{\int_i^v\hat{\varphi}_v(y)m'(y)dy}.
\end{align*}
\end{lemma}
\begin{proof}
Assume that $a<i<v<b$ and let $\bar{X}_t=\{X_t, t<\tau_i\wedge\tau_v\}$ denote the diffusion $X$ killed at the boundaries $i$ and $v$. It is then clear by definition of the processes
$\hat{M}_t$ and $\tilde{I}_t$ that
\begin{align*}
\mathbb{P}_x\left[\hat{M}_T\leq v\right] =\mathbb{P}_x\left[\tilde{I}_T\geq i\right]=\mathbb{P}_x[T<\tau_v\wedge\tau_i]=
1-\frac{\hat{\psi}_i(x)}{\hat{\psi}_i(v)}-\frac{\hat{\varphi}_v(x)}{\hat{\varphi}_v(i)}
\end{align*}
implying that
\begin{align*}
\mathbb{P}_x\left[\hat{M}_T\in dv\right] &=\frac{\hat{\psi}_i(x)}{\hat{\psi}_i^2(v)}\left(\hat{\psi}_i'(v)-BS'(v)\right)dv\\
\mathbb{P}_x\left[\tilde{I}_T\in di\right]&=\frac{\hat{\varphi}_v(x)}{\hat{\varphi}_v^2(i)}\left(-BS'(i)-\hat{\varphi}_v'(i)\right)di.
\end{align*}
On the other hand,
\begin{align*}
\mathbb{P}_x\left[\hat{X}_T\in dy;\hat{M}_T\leq v\right] = \mathbb{P}_x\left[\tilde{X}_T\in dy;\tilde{I}_T\geq i\right] =\mathbb{P}_x\left[\bar{X}_T\in dy\right]=r\bar{G}_r(x,y)dy,
\end{align*}
where
$$
\bar{G}_r(x,y)=\begin{cases}
B^{-1}\hat{\varphi}_v(x)\frac{\hat{\psi}_i(y)}{\hat{\psi}_i(v)} &x\geq y\\
B^{-1} \frac{\hat{\varphi}_v(y)}{\hat{\varphi}_v(i)}\hat{\psi}_i(x)&x\leq y.
\end{cases}
$$
is the Green kernel associated with the killed diffusion $\bar{X}$. Standard differentiation yields
\begin{align*}
\mathbb{P}_x\left[\hat{X}_T\in dy;\hat{M}_T\in dv\right]&=r\frac{\hat{\psi}_i(x)}{\hat{\psi}_i^2(v)}\hat{\psi}_i(y)S'(v)m'(y)dydv\\
\mathbb{P}_x\left[\tilde{X}_T\in dy;\tilde{I}_T\in di\right]&= r\frac{\hat{\varphi}_v(x)}{\hat{\varphi}_v^2(i)}\hat{\varphi}_v(y)S'(i)m'(y)dydi.
\end{align*}
The proposed conditional probability distributions follow from the definition of conditional probability. The alleged conditional expectations are finally obtained by ordinary integration.
\end{proof}

\section{One-boundary, increasing case}\label{sec 1sided}

\subsection{Problem Setting}

Our objective in this section is to delineate the circumstances under which the value of a one-sided threshold policy can be expressed as the expected supremum of a monotonic function
and to identify that function explicitly. In what follows, we will focus on the case where the considered stopping policy can be characterized as a rule where the underlying
process is stopped as soon as it exceeds a given constant threshold. The case where the single boundary stopping rule is to exercise as soon as the underlying falls below a
given constant threshold is completely analogous and, therefore, left untreated.

Let $g:\mathcal{I}\mapsto\R$ be a continuous payoff function for which $g^{-1}(\mathbb{R}_+)\neq\emptyset$ and satisfying
\begin{align}\label{integrability}
\mathbb{E}_x\left[\sup_{t\geq 0}e^{-rt}g(X_t)\right]<\infty
\end{align}
for all $x\in \mathcal{I}$. Assume also that $g\in C^1(\mathcal{I}\setminus \mathcal{P})\cap C^2(\mathcal{I}\setminus \mathcal{P})$, where
$\mathcal{P}\in \mathcal{I}$ is a finite set of points in $\mathcal{I}$ and that $|g'(x\pm)|<\infty$ and $|g''(x\pm)|<\infty$ for all $x\in \mathcal{P}$.

Given the assumed regularity conditions, let $\tau_y=\inf\{t\geq0: X_t\geq y\}$ denote the first exit time of the underlying diffusion from the set $(a, y)$,
where $y\in g^{-1}(\mathbb{R}_+)$. Define now the nonnegative function $V_y:\mathcal{I}\mapsto \mathbb{R}_+$ as
\begin{align}\label{eq Vy}
V_y(x)=\E_x\left[e^{-r\tau_y}g(X_{\tau_y});\tau_y<\infty\right] =\begin{cases}
g(x)\quad &x\geq y\\
\psi(x)\frac{g(y)}{\psi(y)}\quad& x<y.
\end{cases}
\end{align}
Given representation \eqref{eq Vy}, we can now state our identification problem as follows.
\begin{problem}\label{1reuna}
For a given $y\in g^{-1}(\mathbb{R}_+)$, does there exist a nonnegative function $\hat{f}:I\mapsto \mathbb{R}_+$ such that for all $x\in\mathcal{I}$ we would have
\begin{align*}
J_y(x):=\E_x\left[\hat{f}(M_T)\mathbbm{1}_{[y,b)}(M_T)\right]=V_{y}(x),
\end{align*}
where $T\sim\Exp (r)\perp X_t$ (cf. Theorem \ref{theorem abo1}). Under which conditions on the function $\hat{f}$ we have
$$\hat{f}(M_T)\mathbbm{1}_{[y,b)}(M_T) \sup_{t\in[0,T]}\hat{f}(X_t)\mathbbm{1}_{[y,b)}(X_t).$$
\end{problem}
It's worth emphasizing  that Problem \ref{1reuna} is twofold. The first representation problem essentially asks if the
expected value of the exercise payoff accrued at the first hitting time to a constant boundary can be expressed as the expected value of an yet unknown function $\hat{f}$ at
the running maximum of the underlying diffusion at an independent exponentially distributed date. The second question essentially asks when the function $\hat{f}$ is such that the representation agrees with the general functional form utilized in Theorem \ref{theorem abo1}. As we will later establish in this paper, the class of functions satisfying the first representation is strictly larger than the latter.

Before proceeding in the derivation of the representation as an expected supremum, we first establish the following result characterizing the optimal policy. We apply this result later for the
identification of circumstances under which the value of the considered one-sided problem can be expressed as the expected supremum of a monotonic function.
\begin{lemma}\label{Martin1}
Assume that the following conditions are satisfied:
\begin{itemize}
  \item[(i)] there exists a $y^\ast = \argmax\{g(x)/\psi(x)\}\in \mathcal{I}$,
  \item[(ii)] $(\mathcal{G}_rg)(x)\leq 0$ for all $x\in [y^\ast,b)\setminus \mathcal{P}$
  \item[(iii)] $g'(x+)\leq g'(x-)$ for all $x\in[y^\ast,b)\cap \mathcal{P}$
\end{itemize}
Then $V(x)=V_{y^\ast}(x)$ and $\tau_{y^\ast}=\inf\{t\geq 0: X_t\geq y^\ast\}$ is an optimal stopping time.
\end{lemma}
\begin{proof}
It is clear that under our assumptions $V_{y^\ast}(x)$ is nonnegative, continuous, and  dominates the exercise payoff $g(x)$ for all $x\in \mathcal{I}$.
Let $x_0\in(y^\ast,b)\setminus\mathcal{P}$ be a fixed reference point and define the ratio $h_{x_0}(x)=V_{y^\ast}(x)/V_{y^\ast}(x_0)=V_{y^\ast}(x)/g(x_0)$.
It is clear that our assumptions combined with \eqref{linearityGenerator} guarantee that
$$
\sigma_{x_0}^{h_{x_0}}((x,b])=\frac{\psi(x_0)}{Bg(x_0)}\left[\frac{g'(x+)}{S'(x)}\varphi(x)-g(x)\frac{\varphi'(x)}{S'(x)}\right]=
-\frac{\psi(x_0)}{Bg(x_0)}(L_\varphi g)(x+)
$$
is nonnegative and nonincreasing for all $x\geq x_0$. Analogously,
$$
\sigma_{x_0}^{h_{x_0}}([a,x))=
\frac{\varphi(x_0)}{Bg(x_0)}\left[g(x)\frac{\psi'(x)}{S'(x)}-\frac{g'(x-)}{S'(x)}\psi(x)\right]\mathbbm{1}_{(y^\ast,x_0]}(x) =
\frac{\varphi(x_0)}{Bg(x_0)}(L_\psi g)(x-)\mathbbm{1}_{(y^\ast,x_0]}(x)
$$
is nonnegative and nondecreasing for all $x\leq x_0$.  Moreover,
noticing that
$
\sigma_{x_0}^{h_{x_0}}([a,x_0))+\sigma_{x_0}^{h_{x_0}}((x_0,b])=1
$
shows,  by imposing the condition $\sigma_{x_0}^{h_{x_0}}(\{x_0\})=0$, that $\sigma_{x_0}^{h_{x_0}}$ constitutes a probability measure.
Therefore, it induces an $r$-excessive function $h_{x_0}(x)$ via its
Martin representation (cf. Proposition 3.3 in \cite{Salminen1985}). However, since increasing linear transformations of excessive functions are excessive and
$h_{x_0}(x)g(x_0)=V_{y^\ast}(x)$, we observe that $V_{y^\ast}(x)$ constitutes an $r$-excessive majorant of $g$ for $X$. Invoking now \eqref{eq Vy} shows that
$V(x)=V_{y^\ast}(x)$ and consequently, that $\tau_{y^\ast}=\inf\{t\geq 0: X_t\geq y^\ast\}$ is an optimal stopping time.
\end{proof}

\begin{remark}\label{sufficient}
It is at this point worth emphasizing that under the following slightly stricter assumptions there always exists a unique maximizing threshold
$y^\ast = \argmax\{g(x)/\psi(x)\}\in (\tilde{x},b)$ and the conditions of Lemma \ref{Martin1} are satisfied
(cf. Lemma 3.4 in \cite{AlMaRa14}):
\begin{itemize}
\item[(A)] $g^{-1}(\mathbb{R}_+)=(x_0,b)$, where $a<x_0<b$, and $b$ is unattainable for $X$,
\item[(B)]  there exists a $\tilde{x}\in \mathcal{I}$ so that $(\mathcal{G}_rg)(x) > 0$ for all $x\in (a,\tilde{x})\setminus \mathcal{P}$ and
 $(\mathcal{G}_rg)(x)< 0$ for all $x\in (\tilde{x},b)\setminus \mathcal{P}$,
 \item[(C)] $g'(x+)\geq g'(x-)$ for all $x\in(a,\tilde{x})\cap \mathcal{P}$ and $g'(x+)\leq g'(x-)$ for all $x\in[\tilde{x},b)\cap \mathcal{P}$
\end{itemize}
These assumptions are typically met in financial applications of optimal stopping.
Note that these conditions do not impose monotonicity requirements on the behavior of the generator $(\mathcal{G}_rg)(x)$ on $\mathcal{I}\setminus\mathcal{P}$ and only the sign of
$(\mathcal{G}_rg)(x)$ essentially counts. As we will later establish, it is precisely this observation which explains why not all single boundary stopping problems can be represented as expected
suprema.
\end{remark}

\subsection{Characterization of $f$}
Let $y\in g^{-1}(\mathbb{R}_+)$ be given. Utilizing the distribution function characterized in \eqref{eq Pmax} yields
\begin{align*}
J_y(x)=\E_x\left[\hat{f}(M_T)\mathbbm{1}_{[y,b)}(M_T)\right]=\psi(x)\int_{x\lor y}^b\hat{f}(z)\frac{\psi'(z)}{\psi^2(z)}dz.
\end{align*}
Given this expression, it is now sufficient to find a function $\hat{f}$ for which the identity $V_{y}(x)=J_y(x)$ holds. This identity holds for $x\geq y$ provided that
the {\em Volterra integral equation of the the first kind}
\begin{align}
\frac{g(x)}{\psi(x)}=\int_x^b\hat{f}(z)\frac{\psi'(z)}{\psi^2(z)}dz\label{kasvavaid}
\end{align}
is satisfied. Identity \eqref{kasvavaid} has several important implications
both on the regularity of $g$ as well as on the limiting behavior of the ratio $g(x)/\psi(x)$ at $b$. First, we immediately notice that representation \eqref{kasvavaid} implies
that we necessarily need to have $\lim_{x\rightarrow b-}g(x)/\psi(x)=0$. Second,
since the integral of an integrable function is continuous, identity \eqref{kasvavaid}  implies that
the exercise payoff $g(x)$ has to be continuous on $[y,b)$. Moreover, if the unknown function $\hat{f}$ is
continuous outside a finite set of points $\mathcal{P}\in [y,b)$, then identity \eqref{kasvavaid} actually implies
that the exercise payoff $g(x)$ has to be continuously differentiable on $x\in[y,b)\setminus \mathcal{P}$ and possesses both right and left derivatives on $x\in \mathcal{P}$.
Thus, \eqref{kasvavaid} demonstrates that the proposed representation cannot hold unless the exercise payoff $g$ satisfies a set of regularity conditions.

Standard differentiation of identity \eqref{kasvavaid} now shows that
for all $x\in [y,b)\setminus \mathcal{P}$ we have
\begin{align}\label{eq f1}
\hat{f}(x)&=g(x)-\psi(x)\frac{g'(x)}{\psi'(x)} = \frac{S'(x)}{\psi'(x)}(L_\psi g)(x),
\end{align}
coinciding with the function $\rho$ derived in \cite{BaBa} by relying on functional concavity arguments.
Thus, with $\hat{f}(x)$ defined in this way we have, by invoking identity \eqref{kasvavaid} and condition $\lim_{x\rightarrow b-}g(x)/\psi(x)=0$, that
\begin{align}\label{eq Jy}
\begin{aligned}
J_y(x)=\psi(x)\int_{x\lor y}^b\frac{\psi'(z)g(z)-g'(z)\psi(z)}{\psi^2(z)}dz=-\psi(x)\int_{x\lor y}^bd\left(\frac{g(z)}{\psi(z)}\right)
=\psi(x)\frac{g(x\lor y)}{\psi(x\lor y)}
\end{aligned}
\end{align}
for all $x\in [y,b)\setminus \mathcal{P}$. We summarize
these findings in the following theorem.
\begin{theorem}\label{theorem inc}
Fix $y\in g^{-1}(\mathbb{R}_+)$ and let $\hat{f}$ be as in \eqref{eq f1}.
Then, if $\lim_{x\rightarrow b-}g(x)/\psi(x)=0$, we have $J_y(x)=V_y(x)$. Moreover,
if $\hat{f}(x)$ is also nonnegative and nondecreasing for all $x\in[y,b)$, then $V_y(x)$ is $r$-excessive for $X$.
\end{theorem}
\begin{proof}
The first claim follows directly from the derivation of \eqref{eq Jy}. If $\hat{f}$ is also nonnegative and nondecreasing for all $x\in[y,b)$, then $\hat{f}(x)\mathbbm{1}_{[y,b)}(x)$ is
nondecreasing, nonnegative, and upper semicontinuous  on $\mathcal{I}$. In that case $\hat{f}(M_T)\mathbbm{1}_{[y,b)}(M_T)=\sup_{t\in[0,T]}\hat{f}(X_t)\mathbbm{1}_{[y,b)}(X_t)$. Proposition 2.1 in \cite{FoKn1} (see also Lemma 2.2 in \cite{ChSaTa}) then guarantees that
$J_y(x)$ is $r$-excessive for $X$. Since $J_y(x)=V_y(x)$ the alleged result follows.
\end{proof}
Theorem \ref{theorem inc} shows that when $\hat{f}$ is chosen according to the rule \eqref{eq f1}
representation $J_y(x)=V_y(x)$ is valid provided that the limiting condition $\lim_{x\rightarrow b-}g(x)/\psi(x)=0$ is met.
Moreover, Theorem \ref{theorem inc} also shows that if $\hat{f}(x)\mathbbm{1}_{[y,b)}(x)$ is also nondecreasing and nonnegative, then the
representation is $r$-excessive for the underlying diffusion $X$. Note, however, that the representation needs not to majorize the exercise payoff and, therefore,
it does not necessarily coincide with the value of the considered stopping problem. Moreover, the monotonicity and nonnegativity of  $\hat{f}(x)\mathbbm{1}_{[y,b)}(x)$
is sufficient but {\em not necessary}  for the $r$-excessivity of $J_y(x)$. As we will later see, there are circumstances where $J_y(x)$ is $r$-excessive even
when $\hat{f}(x)\mathbbm{1}_{[y,b)}(x)$ is not monotonic.

We are now in position to establish the following.
\begin{theorem}\label{meidan esitys1}
Assume that the conditions of Lemma \ref{Martin1} are satisfied and that $\lim_{x\rightarrow b}g(x)/\psi(x)=0$. Then,
$$
V(x)=V_{y^\ast}(x)=J_{y^\ast}(x)=\E_x\left[\hat{f}(M_T)\mathbbm{1}_{[y^\ast,b)}(M_T)\right].
$$
\end{theorem}
\begin{proof}
It is clear that the conditions of the first claim of Theorem \ref{theorem inc} are satisfied. Consequently, $J_{y^\ast}(x)=V_{y^\ast}(x)$. The alleged result now follows from Lemma \ref{Martin1}.
\end{proof}
Theorem \ref{meidan esitys1} proves that the value of the optimal stopping strategy can be expressed as the expected value of a mapping $\hat{f}$ at the running maximum of the underlying diffusion. This does not yet guarantee that the value of the stopping could be expressed as an expected supremum.
In what follows, our objective is to first determine a set of conditions under which we also have that
$J_y(x)=\E_x\left[\sup_{t\in[0,T]}\hat{f}(X_t)\mathbbm{1}_{[y,b)}(X_t)\right]$. In order to accomplish that objective, we first present an auxiliary result characterizing the circumstances
under which the function $\hat{f}$ is indeed monotonic.
\begin{lemma}\label{lemma mon1}
Let $y\in g^{-1}(\mathbb{R}_+)$ be given. Assume that either
\begin{itemize}
 \item[(A)] $g(x)$ is concave and $\psi(x)$ is convex on $[y,b)$, or
 \item[(B)]  there is a $z\in (a,y)$ so that $g(x)/\psi(x)$ is locally increasing at $z$,  $g'(x+)\leq g'(x-)$ for all $x\in (z,b)\cap\mathcal{P}$, and
 $(\mathcal{G}_rg)(x)$ is non-increasing and non-positive for all $x\in (z,b)$.
 \end{itemize}
Then, the function  $\hat{f}(x)$ characterized by \eqref{eq f1} is non-decreasing on $[y,b)$.
\end{lemma}
\begin{proof}
It is clear from \eqref{eq f1} that the required monotonicity of $\hat{f}$ is met provided that inequality
\begin{align}
\frac{d}{dx}\left(\frac{g'(x)}{\psi'(x)}\right)&<0\label{eq nec f1}
\end{align}
is satisfied for all $x\in [y,b)\setminus \mathcal{P}$ and
\begin{align}
\hat{f}(x+) -\hat{f}(x-) = \frac{g'(x-)-g'(x+)}{\psi'(x)}> 0\label{eq nec f2}
\end{align}
for all $x\in [y,b)\cap \mathcal{P}$. First, if $g$ is concave and $\psi$ is convex on $[y,b)$, then the inequalities \eqref{eq nec f1} and
\eqref{eq nec f2} are satisfied and $g'(x)/\psi'(x)$ is non-increasing on $[y,b)$ as claimed.
Assume now instead that the conditions of part (B) are satisfied. It is clear that since $[y,b)\subset(z,b)$ \eqref{eq nec f2} is satisfied by assumption
for all $x\in [y,b)\cap\mathcal{P}$. On the other hand, standard differentiation shows that for all $x\in(z,b)\setminus \mathcal{P}$
\begin{align*}
\frac{d}{dx}\left(\frac{g'(x)}{\psi'(x)}\right)=\frac{S'(x)}{{\psi'}^2(x)}\left[\frac{g''(x)}{S'(x)}\psi'(x)-\frac{\psi''(x)}{S'(x)}g'(x)\right]
=\frac{2S'(x)\mathcal{D}(x)}{\sigma^2(x){\psi'}^2(x)}.
\end{align*}
where
$$
\mathcal{D}(x)=(\mathcal{G}_rg)(x)\frac{\psi'(x)}{S'(x)}+r(L_\psi g)(x).
$$
The assumed monotonicity and non-positivity of $(\mathcal{G}_rg)(x)$ on $(z,b)\setminus \mathcal{P}$ now implies that
\begin{align*}
\mathcal{D}(x)&=(\mathcal{G}_rg)(x)\frac{\psi'(x)}{S'(x)}-r\int_z^x\psi(t)(\mathcal{G}_rg)(t)m'(t)dt+r(L_\psi g)(z+)\\
&\leq (\mathcal{G}_rg)(x)\frac{\psi'(z)}{S'(z)}+r(L_\psi g)(z+)\leq r(L_\psi g)(z+)
\end{align*}
for all $x\in(z,b)\setminus \mathcal{P}$. However, the assumed monotonicity of $g(x)/\psi(x)$ in a neighborhood of $z$ then guarantees that
$(L_\psi g)(z+) \leq 0$, proving that $\mathcal{D}(x)\leq 0$ for all $x\in(z,b)\setminus \mathcal{P}$.
\end{proof}
Lemma \ref{lemma mon1} states a set of conditions under which the function  $\hat{f}(x)$ characterized by \eqref{eq f1} is non-decreasing on the set $[y,b)$ and, therefore, the function
$\hat{f}(x)\mathbbm{1}_{[y,b)}(x)$ is nondecreasing on $\mathcal{I}$. Interestingly, the first of these conditions is based solely on the concavity of the exercise payoff and the convexity of the increasing fundamental
solution without imposing further requirements. A sufficient condition for the
convexity of the fundamental solution $\psi(x)$ is that $\mu(x)-rx$ is non-increasing on $\mathcal{I}$ and $a$ is unattainable for the underlying diffusion (see \cite{Alvarez2003}).
Consequently, part (A) of Lemma \ref{lemma mon1} essentially delineates circumstances under which the monotonicity
of function $\hat{f}(x)$ could be, in principle,
characterized solely based on the infinitesimal characteristics of the underlying diffusion and the concavity of the exercise payoff.
Part (B) of Lemma \ref{lemma mon1} shows, in turn, how the monotonicity of the  function  $\hat{f}(x)$ is associated with the monotonicity of the generator $(\mathcal{G}_rg)(x)$.
The conditions of part (B) of Lemma \ref{lemma mon1} are satisfied, for example,  under the assumptions of Remark \ref{sufficient} provided that
$(\mathcal{G}_rg)(x)$ is non-increasing on $(\tilde{x},b)$ and $z\in (\tilde{x}, y\wedge y^\ast)$.

Moreover, it is clear that under the conditions of Lemma \ref{lemma mon1} we have $J_y(x)=V_y(x)$ for all $y\in\mathcal{I}$.
However, without imposing further restrictions on the behavior of the payoff we do not know whether $\hat{f}(x)\mathbbm{1}_{[y,b)}(x)$ generates
the smallest $r$-excessive majorant of the exercise payoff $g$ or not, nor do we know how $\hat{f}(x)\mathbbm{1}_{[y,b)}(x)$ behaves in the neighborhood of the optimal
stopping boundary. Our next theorem summarizes a set of conditions under which these questions can be unambiguously answered.
\begin{theorem}\label{thm optimal condition}
Define $y^\ast = \inf\{y: \hat{f}(y)\geq 0\}$ and assume that the conditions (A) or (B) of Lemma \ref{lemma mon1} are satisfied on $[y^\ast,b)$.
Then, $y^\ast = \argmax\{g(x)/\psi(x)\}\in \mathcal{I}$. Especially, $\hat{f}(y^\ast)=0$ if $y^\ast\in \mathcal{I}\setminus \mathcal{P}$ and
$$
\hat{f}(y^\ast) = g(y^\ast)-\frac{\psi(y^\ast)}{\psi'(y^\ast)}g'(y^\ast+)> 0
$$ if $y^\ast\in \mathcal{P}$. Moreover,
\begin{align}
\hat{f}(x) = \frac{S'(x)}{\psi'(x)}(L_\psi g)(x+) = \frac{(L_\psi g)(y^\ast+)-\int_{y^\ast}^{x}(\mathcal{G}_rg)(z)\psi(z)m'(z)dz}{r\int_{a}^{x}\psi(z)m'(z)dz+\psi'(a+)/S'(a+)}\label{esitys1}
\end{align}
for all $x\in (y^\ast,b)\setminus \mathcal{P}$, and
$$
V(x)=V_{y^{\ast}}(x)=J_{y^\ast}(x)=\psi(x)\sup_{y\geq x}\left[\frac{g(y)}{\psi(y)}\right]=\psi(x)\frac{g(x\lor y^\ast)}{\psi(x\lor y^\ast)}=
\E_x\left[\sup_{t\in[0,T]}\hat{f}(X_t)\mathbbm{1}_{[y^\ast,b)}(X_t)\right].
$$
\end{theorem}
\begin{proof}
We first observe that condition (A) or (B) of Lemma \ref{lemma mon1} guarantee that $\hat{f}(x)$ is nondecreasing on $[y^\ast,b)$. However, since
$$
\frac{d}{dx}\left(\frac{g(x)}{\psi(x)}\right)=-\frac{\psi'(x)}{\psi^2(x)}\hat{f}(x),
$$
and the ratio $g(x)/\psi(x)$ is continuous, we notice that $g(x)/\psi(x)$ is increasing on $(a,y^\ast)$ and decreasing on $(y^\ast,b)$. Consequently, $y^\ast = \argmax\{g(x)/\psi(x)\}$.
As is clear, if $y^\ast \in\mathcal{I}\setminus \mathcal{P}$, then
we necessarily have $g'(y^\ast)\psi(y^\ast)=g(y^\ast)\psi'(y^\ast)$ showing that $\hat{f}(y^\ast)=0$ in that case. If the optimum is, however, attained on $\mathcal{P}$, then we necessarily have
that $g'(y^\ast-)\psi(y^\ast)\geq g(y^\ast)\psi'(y^\ast)\geq g'(y^\ast+)\psi(y^\ast)$, where at least one of the inequalities is strict, proving that $\hat{f}(y^\ast+)>0$ in that case.

The last claim follows from the identity $\hat{f}(x)=\frac{S'(x)}{\psi'(x)}(L_\psi g)(x)$ by invoking the canonical form
$$
\frac{\psi'(x)}{S'(x)}-\frac{\psi'(a+)}{S'(a+)}=r\int_a^x\psi(z)m'(z)dz
$$
and noticing that
$$
(L_\psi g)'(x)=-(\mathcal{G}_rg)(x)\psi(x)m'(x)
$$
for all $x\in \mathcal{I}\setminus \mathcal{P}$. Finally, identity $V_{y^{\ast}}(x)=J_{y^\ast}(x)=V(x)$ follows from Theorem \ref{theorem inc} after noticing that
identity $y^\ast=\argmax\{g(x)/\psi(x)\}$ guarantees that the proposed value dominates the exercise payoff.
\end{proof}
Theorem \ref{thm optimal condition} shows that the continuity of the function $\hat{f}$ at the optimal boundary $y^\ast$ coincides with the standard
{\em smooth fit principle} requiring that the value should be continuously differentiable across the optimal boundary. However, as is clear from
Theorem \ref{thm optimal condition}, if the optimal boundary is attained at a threshold where the exercise payoff is not differentiable, then $\hat{f}$
is discontinuous at the optimal boundary $y^\ast$. Furthermore, since the nonnegativity and monotonicity of $\hat{f}(x)\mathbbm{1}_{[y^\ast,b)}(x)$ on $[y^\ast,b)$ are sufficient
for the validity of Theorem \ref{thm optimal condition}, we observe in accordance with the results by \cite{ChSaTa} that $\hat{f}(x)\mathbbm{1}_{[y^\ast,b)}(x)$ is
only upper semicontinuous on $\mathcal{I}$.

Theorem \ref{thm optimal condition} also shows that $\hat{f}(x)$ has a neat integral representation \eqref{esitys1} capturing the size of the potential
discontinuity of $\hat{f}(x)$ at $y^\ast$. In the case where $a$ is unattainable and the smooth fit principle is satisfied at $y^\ast$ \eqref{esitys1}
can be re-expressed as  (cf. Proposition 2.13 in \cite{ChSaTa})
\begin{align}
\hat{f}(x) = \frac{\int_{y^\ast}^{x}(\mathcal{G}_rg)(z)\psi(z)m'(z)dz}{r\int_{a}^{x}\psi(z)m'(z)dz}\label{esitys1b}
\end{align}
and, hence,
\begin{align}
V(x) = \mathbb{E}_x\left[\frac{\int_{y^\ast}^{M_T}(\mathcal{G}_rg)(z)\psi(z)m'(z)dz}{r\int_{a}^{M_T}\psi(z)m'(z)dz}\mathbbm{1}_{[y^\ast,b)}(M_T)\right]\label{esitys1c}
\end{align}
Finally, it is clear that if the sufficient conditions stated in Remark \ref{sufficient} are satisfied, and in addition $(\mathcal{G}_rg)(x)$ is non-increasing on $(y^\ast,b)$,
and $a$ is unattainable for the underlying diffusion, then the conditions of  Theorem \ref{thm optimal condition} are met and
\begin{align*}
V(x) = \mathbb{E}_x\left[\sup_{t\in [0,T]}\left(\frac{\int_{y^\ast}^{X_t}(\mathcal{G}_rg)(z)\psi(z)m'(z)dz}{r\int_{a}^{X_t}\psi(z)m'(z)dz}\mathbbm{1}_{[y^\ast,b)}(X_t)\right)\right].
\end{align*}

\subsection{Examples}
We now illustrate our general findings in two separate examples. The first example focuses on a case where the payoff is smooth and the stopping strategy is of
the single boundary type. Despite these favorable properties, we will show that it does not always result into a value characterizable as an expected supremum.
The second example, in turn, focuses on a less smooth case resulting into a representation where the function $f(x)\mathbbm{1}_{[y,b)}(x)$ is monotone but not everywhere continuous.
\subsubsection{Example 1: Smooth Payoff}
In order to illustrate our findings we now assume that the upper boundary $b$ is unattainable for $X$ and that the exercise payoff can be expressed as an expected cumulative present value
$g(x) = (R_r\pi)(x)$
for some continuous revenue flow $\pi\in \mathcal{L}^1(\mathcal{I})$ satisfying the conditions $\pi(x)\gtreqqless 0$ for $x \gtreqqless x_0$, where $x_0\in (a,b)$,
$\lim_{x\downarrow a}\pi(x)< -\varepsilon$ and
$\lim_{x\uparrow b}\pi(x) > \varepsilon$ for some $\varepsilon > 0$.

It is clear that under these conditions the exercise payoff satisfies the conditions $g\in C^2(\mathcal{I})$ and
$(\mathcal{G}_rg)(x)= -\pi(x)\lesseqqgtr 0$ for $x \gtreqqless x_0$. Moreover, utilizing representation \eqref{Green} shows that
under our assumptions
\begin{align*}
(L_\psi g)(x)=\int_a^x\psi(t)\pi(t)m'(t)dt
\end{align*}
It is clear from our assumption that $(L_\psi g)(x)<0$ for all $x\leq x_0$ and $(L_\psi g)(x)$ is monotonically increasing on $(x_0,b)$. Fix $x_1>x_0$. Then a
standard application of the mean value theorem
yields
$$
(L_\psi g)(x)=(L_\psi g)(x_1)+\int_{x_1}^x\psi(t)\pi(t)m'(t)dt = (L_\psi g)(x_1)+\frac{\pi(\xi)}{r}\left[\frac{\psi'(x)}{S'(x)}-\frac{\psi'(x_1)}{S'(x_1)}\right],
$$
where $\xi\in(x_1,x)$. Letting $x\rightarrow b$ and noticing that $\psi'(x)/S'(x)\rightarrow \infty$ as $x\rightarrow b$ (since $b$ was assumed to be unattainable
for $X$, cf. p. 19 in \cite{BorSal02}) then shows that
$\lim_{x\uparrow b}(L_\psi g)(x)=\infty$ proving that equation $(L_\psi g)(x)=0$ has a unique root $y^\ast\in(x_0,b)$ and that $y^\ast = \argmax\{(R_r\pi)(x)/\psi(x)\}$. Moreover, the value
\eqref{eq prob} can be expressed as
$$
V(x) = \psi(x)\sup_{y\geq x}\left[\frac{(R_r\pi)(x)}{\psi(x)}\right] = \begin{cases}
(R_r\pi)(x) &x\geq y^\ast\\
\frac{(R_r\pi)(y^\ast)}{\psi(y^\ast)}\psi(x) &x<y^\ast.
\end{cases}
$$
It is clear that under our assumptions the function $f(x)$ characterized in Theorem \ref{theorem inc} can be expressed  as
$$
f(x)=\frac{S'(x)}{\psi'(x)}\int_{a}^x\psi(y)\pi(y)m'(y)dy.
$$
As was established in Theorem \ref{thm optimal condition}, we have that $f(y^\ast)=0$ and, therefore,
$$
V(x)=\E_x\left[\frac{S'(M_T)}{\psi'(M_T)}\int_{y^{\ast}}^{M_T}\psi(y)\pi(y)m'(y)dy\mathbbm{1}_{[y^{\ast},b)}(M_T)\right] = \psi(x)\frac{(R_r\pi)(y^\ast\lor x)}{\psi(y^\ast\lor x)}.
$$
Moreover,
standard differentiation now shows that for all $x\in(y^\ast,b)$ we have
$$
f'(x) = \frac{2S'(x)\psi(x)}{{\psi'}^2(x)\sigma^2(x)}\left[\pi(x)\frac{\psi'(x)}{S'(x)}-r\int_{y^\ast}^{x}\psi(t)\pi(t)m'(t)dt\right]
$$
demonstrating that $f$ is nondecreasing for $x\in(y^\ast,b)$ only if
$$
\pi(x)\frac{\psi'(x)}{S'(x)}\geq r\int_{y^\ast}^{x}\psi(t)\pi(t)m'(t)dt
$$
for all $x\geq y^\ast$. Otherwise it is clear from our results that the value of the considered optimal stopping problem {\em cannot} be expressed as an expected supremum
(see Figure \ref{fig example 1}(A)).
A simple sufficient condition guaranteeing the required monotonicity is to assume that $\pi(x)$ is nondecreasing on $(x_0,b)$ since in that case we have
$$
f'(x)\geq\frac{2S'(x)\psi(x)}{{\psi'}^2(x)\sigma^2(x)}\pi(x)\frac{\psi'(y^\ast)}{S'(y^\ast)}\geq 0.
$$
If this is indeed the case, then $\sup_{t\in[0,T]}f(X_t)\mathbbm{1}_{[y^\ast,b)}(X_t) = f(M_T)\mathbbm{1}_{[y^\ast,b)}(M_T)$ and
$$V(x) = \mathbb{E}_x\left[\sup_{t\in[0,T]}f(X_t)\mathbbm{1}_{[y^\ast,b)}(X_t)\right] = \mathbb{E}_x\left[f(M_T)\mathbbm{1}_{[y^\ast,b)}(M_T)\right].$$

\subsubsection{Example 2: Capped Call Option}
In order to illustrate our findings in a nondifferentiable setting, assume now that the upper boundary $b$ is unattainable for $X$ and that the exercise
payoff $g(x) = \min((x-K)^+, C)$ (a {\em capped call option}), with $a < K < C < K+C < b$, satisfies the limiting inequality
\begin{align}
\lim_{x\downarrow a}\frac{|x-K|}{\varphi(x)}<\infty.\label{limit}
\end{align}
Assume also that the
appreciation rate $\theta(x)=\mu(x)-r(x-K)$ satisfies the conditions $\theta\in\mathcal{L}^1_r(\mathcal{I})$, $\theta(x)\gtreqqless 0$ for $x\lesseqqgtr x_0^\theta$,
where $x_0^\theta\in \mathcal{I}$,
and $\lim_{x\rightarrow b}\theta(x) < -\varepsilon$ for $\varepsilon >0$.

It is now clear that the conditions of Remark \ref{sufficient} are satisfied. Thus, we known that there exists a unique optimal exercise threshold
$x^\ast =\argmax\{g(x)/\psi(x)\}$ and $V(x)=V_{x^\ast}(x)$. Our objective is now to prove that
this threshold reads as
$x^\ast = \min(C+K,y^\ast)$,
where $y^\ast > x_0^\theta$ is the unique root of equation
$$
\int_a^{y^\ast}\psi(y)\theta(y)m'(y)dy=\frac{a-K}{\varphi(a)}.
$$
To see that this is indeed the case, we first observe by applying part (A) of Corollary 3.2 in \cite{Al04} combined with the limiting condition \eqref{limit} that
$$
\frac{\psi^2(x)}{S'(x)}\frac{d}{dx}\left[\frac{x-K}{\psi(x)}\right] =\frac{\psi(x)}{S'(x)}-(x-K)\frac{\psi'(x)}{S'(x)}=\int_a^x\psi(t)\theta(t)m'(t)dt-\frac{a-K}{\varphi(a)}.
$$
Applying analogous arguments with the ones in Example 1, we find that equation
$$
\int_a^x\psi(t)\theta(t)m'(t)dt-\frac{a-K}{\varphi(a)}=0
$$
has a unique root $y^\ast\in (x_0^\theta,b)$ so that $y^\ast=\argmax\{(x-K)/\psi(x)\}$. Moreover,
$$
U(x)=\sup_{\tau}\mathbb{E}_x\left[e^{-r\tau}(X_\tau-K)^+\right]=\begin{cases}
x-K &x\geq y^\ast\\
(y^\ast-K)\frac{\psi(x)}{\psi(y^\ast)} &x<y^\ast.
\end{cases}
$$

In light of these observations, we find that if $y^\ast\in (K,K+C)$, then it is sufficient to notice that $V_{x^\ast}(x)=\min(C,U(x))$ is $r$-excessive since
constants are $r$-excessive and $U(x)$ is also $r$-excessive. Moreover, since
both $C$ and $U(x)$ dominate the payoff, we notice that $V_{x^\ast}(x)=\min(C,U(x))$ constitutes the smallest $r$-excessive majorant of $g(x)$ and, therefore, $V(x)=V_{x^\ast}(x)=\min(C,U(x))$.
If instead $y^\ast\geq K+C$, then $x^\ast=K+C=\argmax\{g(x)/\psi(x)\}$ and the optimal policy is to follow the stopping policy $\tau_{x^\ast}=\inf\{t\geq 0: X_t\geq K+C\}$ with a value
$$
\tilde{U}(x)=C\mathbb{E}_x\left[e^{-r\tau_{x^\ast}}\right]=\begin{cases}
C &x\geq C+K\\
C\frac{\psi(x)}{\psi(C+K)} &x<C+K.
\end{cases}
$$
Given these findings, we notice that if $y^\ast\geq K+C$, then
$$
f(x) = C\mathbbm{1}_{[x^\ast,b)}(x)\geq 0
$$
is nonnegative and nondecreasing and, consequently,
$$
V(x)=C\mathbb{E}_x\left[\mathbbm{1}_{[x^\ast,b)}(M_T)\right]=C\mathbb{P}_x\left[M_T\geq K+C\right].
$$
However, since $f(x^\ast-)=0$ and $f(x^\ast+)=C$ we notice that $f$ is discontinuous at the optimal threshold $x^\ast$ (see Figure \ref{fig example 1}(B)).
If $y^\ast<K+C$, then the nonnegative function
$$
f(x)=\begin{cases}
C &x\geq C+K\\
x-K-\frac{\psi(x)}{\psi'(x)} &x\in[y^\ast,K+C)
\end{cases}
$$
in nondecreasing only if the increasing fundamental solution is convex on $(y^\ast,K+C)$ (it has to be locally convex at $y^\ast$). If the convexity requirement is met, then
$$
V(x)=\mathbb{E}_x\left[\left(M_T-K-\frac{\psi(M_T)}{\psi'(M_T)}\right)\mathbbm{1}_{[y^\ast,C+K)}(M_T)\right]+C\mathbb{P}_x\left[M_T\geq C+K\right].
$$
Moreover, since
$f(C+K+)=C > C-\frac{\psi(C+K-)}{\psi'(C+K-)} = f(C+K-)$, we notice that $f$ is discontinuous at $C+K$.

\begin{figure}[!ht]
\begin{center}
\begin{subfigure}[b]{0.4\textwidth}
\begin{center}
\includegraphics[width=\textwidth]{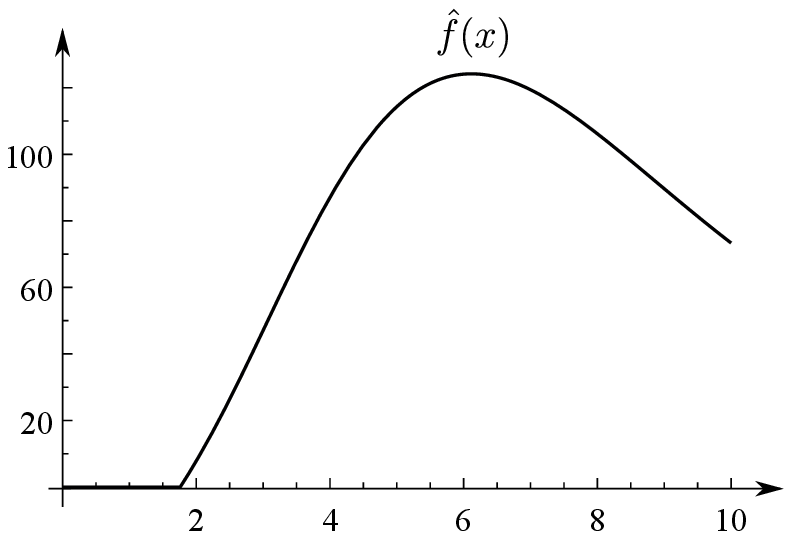}
\end{center}
\caption{\tiny Example 1: Smooth payoff with $\pi(x)=(x^5-2)e^{-x}+1$ leads to a non-increasing $\hat{f}$. In this case the representation as an expected supremum fails to exist.}\label{fig example 1a}
\end{subfigure}
\qquad
\begin{subfigure}[b]{0.4\textwidth}
\begin{center}
\includegraphics[width=\textwidth]{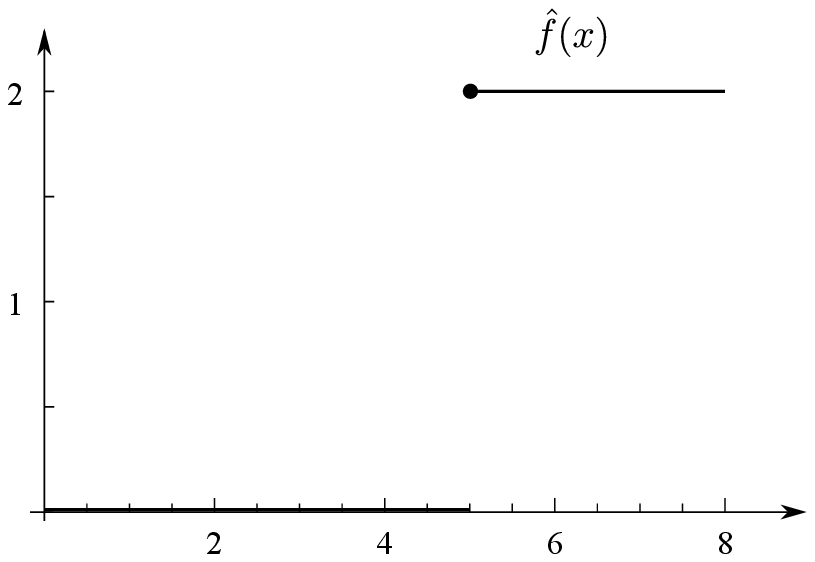}
\end{center}
\caption{\tiny Example 2: Capped call option with $g(x)=\min\{(x-3)^+,2\}$ leads to a discontinuous $\hat{f}$.\\ }\label{fig example 1b}
\end{subfigure}
\end{center}
\caption{\small Numerical examples based on geometric Brownian motion. Parameters have been chosen such that $\psi=x^2$ and $\varphi=x^{-4}$}\label{fig example 1}
\end{figure}

\section{Two-boundary case}

Having considered the one-sided stopping policies our objective is to now extend our analysis to a two-boundary setting and determine a representation of the value in terms of a supremum
of a given function satisfying a set of regularity and monotonicity conditions. In order to accomplish this task,
we assume throughout this section that $g:\mathcal{I}\mapsto\R$ be a continuous payoff function for which $g^{-1}(\mathbb{R}_+)\neq\emptyset$ and satisfying condition
\begin{align}\label{integrability2}
\mathbb{E}_x\left[\left(\sup_{t\geq 0}e^{-rt}g(X_t)\right)\lor\left(-\inf_{t\geq 0}e^{-rt}g(X_t)\right)\right]<\infty
\end{align}
for all $x\in \mathcal{I}$. Along the lines of the single boundary setting we also assume that $g\in C^1(\mathcal{I}\setminus \mathcal{P})\cap C^2(\mathcal{I}\setminus \mathcal{P})$, where
$\mathcal{P}\in \mathcal{I}$ is a finite set of points in $\mathcal{I}$ and that $|g'(x\pm)|<\infty$ and $|g''(x\pm)|<\infty$ for all $x\in \mathcal{P}$.

Let $\tau_{z,y}=\inf\{t\geq0: X_t\notin (z,y)\}$ denote the first exit time of $X$ from the open set $(z,y)\subset \mathcal{I}$ with compact closure
in $\mathcal{I}$ and denote by $$V_{z,y}(x):=\E_x\left[e^{-r\tau_{z,y}}g(X_{\tau_{z,y}});\tau_{z,y}<\infty\right]$$ the expected present value of the exercise payoff accrued
from following that stopping strategy. It is well known that in that case
$V$ can be rewritten as (cf. \cite{Lempa10})
\begin{align}\label{eq Vzy}
V_{z,y}(x)=\begin{cases}
g(x)\quad &x\in(a,z]\cup[y,b)\\
\frac{\hat{\varphi}_y(x)}{\hat{\varphi}_y(z)}g(z)+\frac{\hat{\psi}_z(x)}{\hat{\psi}_z(y)}g(y)\quad & x\in(z,y),
\end{cases}
\end{align}
where $\hat{\varphi}_y(x)=\varphi(x)\psi(y)-\varphi(y)\psi(x)$ denotes the decreasing and $\hat{\psi}_z(x)=\psi(x)\varphi(z)-\psi(z)\varphi(x)$ the increasing
fundamental solution of the ordinary differential equation $(\mathcal{G}_ru)(x)=0$ defined with respect to the killed diffusion $\{X_t; t\in [0,\tau_{z,y})\}$.
Within this two-boundary setting our identification problem can be stated as follows:
\begin{problem}\label{repre2bound}
For a given pair $z,y\in g^{-1}(\mathbb{R}_+)$ satisfying the condition $a<z<y<b$, is there a function
$f(x)=f_1(x)\mathbbm{1}_{(a,z]}(x)+f_2(x)\mathbbm{1}_{[y,b)}(x)$, where $f_1(x)$ is nonincreasing and $f_2(x)$ is nondecreasing
such that for all $x\in\mathcal{I}$ we would have
\begin{align}
J_{(z,y)}(x):=\E_x\left[f_1(I_T)\mathbbm{1}_{(a,z]}(I_T)\lor f_2(M_T)\mathbbm{1}_{[y,b)}(M_T)\right]=V_{z,y}(x),\label{decomp}
\end{align}
where $T\sim \Exp (r)$ is independent of the underlying $X$.
\end{problem}
It is at this point worth pointing out that if $f_1(z-)\wedge f_2(y+)\geq 0$, then we clearly have
$$
\sup\{f(X_t); t\leq T\} = f_1(I_T)\mathbbm{1}_{(a,z]}(I_T)\lor f_2(M_T)\mathbbm{1}_{[y,b)}(M_T)
$$
Consequently, Problem \ref{repre2bound} essentially asks if there exists a function such that the expected present value
of the payoff accrued at the first exit time from an open interval can be expressed as as an expected supremum of that particular function or not.
Especially, if the inequality $f_1(z-)\wedge f_2(y+)\geq 0$ is satisfied, then we find by applying Jensen's inequality that
\begin{align}
\E_x\left[\sup\{f(X_t);t\leq T\}\right]\geq
\left(\psi(x)\int_{x\lor y}^bf_2(t)\frac{\psi'(t)}{\psi^2(t)}dt\right) \lor \left(-\varphi(x)\int_a^{x\land z}f_1(t)\frac{\varphi'(t)}{\varphi^2(t)}dt\right).\label{Jensen}
\end{align}
Therefore, whenever $V_{z,y}(x)$ can be expressed as an expected supremum, it has to dominate the lower bound \eqref{Jensen}.

On the other hand, the function $f$ stated in Problem \ref{repre2bound} has an additive form. One could, thus, be tempted to search for a similar additive
representation of the supremum. Unfortunately, such an approach is not possible since the assumed monotonicity of the functions $f_1$ and $f_2$ implies that
\begin{align*}
\sup\{f(X_t); t\leq T\}&\leq \sup\{f_1(X_t)\mathbbm{1}_{(a,z]}(X_t); t\leq T\} + \sup\{f_2(X_t)\mathbbm{1}_{[y,b)}(X_t); t\leq T\} \\
&\leq   \sup\{f_1^+(X_t)\mathbbm{1}_{(a,z]}(X_t); t\leq T\} + \sup\{f_2^+(X_t)\mathbbm{1}_{[y,b)}(X_t); t\leq T\}\\
& =
f_1^+(I_T)\mathbbm{1}_{(a,z]}(I_T) + f_2^+(M_T)\mathbbm{1}_{[y,b)}(M_T).
\end{align*}
Thus, if the inequality $f_1(z-)\wedge f_2(y+)\geq 0$ is met, we observe that
$$
\sup\{f(X_t); t\leq T\} \leq f_1(I_T)\mathbbm{1}_{(a,z]}(I_T) + f_2(M_T)\mathbbm{1}_{[y,b)}(M_T)
$$
and, therefore, that
\begin{align}
\E_x\left[\sup\{f(X_t);t\leq T\}\right]\leq \psi(x)\int_{x\lor y}^bf_2(t)\frac{\psi'(t)}{\psi^2(t)}dt -\varphi(x)\int_a^{x\land z}f_1(t)\frac{\varphi'(t)}{\varphi^2(t)}dt.\label{ineq}
\end{align}
Based on these findings, we can establish the following.
\begin{lemma}\label{lemma J}
Assume that $f_1(z-)\lor f_2(y+)\geq0$. Then $J_{(z,y)}(x)$ is $r$-excessive for the underlying diffusion $X$. Moreover, if $f_1(z-)\wedge f_2(y+)\geq 0$, then
$J_{(z,y)}(x) = \E_x\left[\sup\{f(X_t);t\leq T\}\right]$ satisfies inequality \eqref{ineq}
for all $x\in \mathcal{I}$.
\end{lemma}
\begin{proof}
We first observe that if $f_1(z-)\lor f_2(y+)\geq0$, then $f_1(x)\mathbbm{1}_{(a,z]}(x)\lor f_2(x)\mathbbm{1}_{[y,b)}(x)$ is nonnegative and upper semicontinuous
for all $x\in \mathcal{I}$. Proposition 2.1 in \cite{FoKn1} then implies that $J_{(z,y)}(x)$
is $r$-excessive for $X$. The second claim was proven in the text.
\end{proof}
Lemma \ref{lemma J} states a set of easily verifiable conditions characterizing circumstances under which the proposed representation is $r$-excessive for the underlying $X$.
It is, however, worth noticing that Lemma \ref{lemma J} does not make statements on the relationship between the values $J_{(z,y)}(x)$ and
$V_{z,y}(x)$. Thus, characterizing the expected value $J_{(z,y)}(x)$ without an explicit characterization of the functions $f_1$ and $f_2$ is not possible
and more analysis is needed. It is also worth emphasizing that Lemma \ref{lemma J} shows that if the auxiliary functions $f_1$ and $f_2$ are nonnegative on $\mathcal{I}$, then the expected supremum
$J_{(z,y)}(x)$ is bounded from above by a functional form which, in principle, could be  computed explicitly provided that the functions $f_1$ and $f_2$ were known.

By reordering terms, the value \eqref{eq Vzy} can also be expressed as
\begin{equation}\label{eq A}
V_{z,y}(x)=A_1(z,y)\varphi(x)+A_2(z,y)\psi(x),
\end{equation}
where
$$
A_1(z,y)=\frac{\psi(y)g(z)-g(y)\psi(z)}{\psi(y)\varphi(z)-\psi(z)\varphi(y)}
$$
and
$$
A_2(z,y)=\frac{\varphi(z)g(y)-g(z)\varphi(y)}{\psi(y)\varphi(z)-\psi(z)\varphi(y)}.
$$
Hence, if the exercise payoff is differentiable at the thresholds $z$ and $y$, then
\begin{align}
\frac{1}{\psi(z)}\frac{\partial A_1}{\partial y}(z,y)=-\frac{1}{\varphi(z)}\frac{\partial A_2}{\partial y}(z,y)=
\frac{S'(y)}{\hat{\psi}_z^2(y)}\left[g(y)\frac{\hat{\psi}_z'(y)}{S'(y)}-\hat{\psi}_z(y)\frac{g'(y)}{S'(y)}-Bg(z)\right]\label{nc1}
\end{align}
and
\begin{align}
\frac{1}{\psi(y)}\frac{\partial A_1}{\partial z}(z,y)=-\frac{1}{\varphi(y)}\frac{\partial A_2}{\partial z}(z,y)=
\frac{S'(z)}{\hat{\psi}_z^2(y)}\left[\hat{\varphi}_y(z)\frac{g'(z)}{S'(z)}-g(z)\frac{\hat{\varphi}_y'(z)}{S'(z)}-Bg(y)\right].\label{nc2}
\end{align}
We will apply these results later when deriving the auxiliary mappings needed for the representation of the value as an expected supremum.
Before proceeding in our analysis, we first state the following auxiliary lemma:
\begin{lemma}\label{Martin2}
Assume that the following conditions are satisfied:
\begin{itemize}
  \item[(i)] there exists a unique pair $(z^\ast,y^\ast)$ satisfying the inequality $a<z^\ast < y^\ast<b$ such that $V_{z^\ast,y^\ast}(x)=\sup_{z,y\in \mathcal{I}}V_{z,y}(x)$,
  \item[(ii)] $\lim_{x\rightarrow a+}(g'(x)\psi(x)-g(x)\psi'(x))\leq0$ and $\lim_{x\rightarrow b-}(g'(x)\varphi(x)-g(x)\varphi'(x))\geq0$,
  \item[(iii)] $(\mathcal{G}_rg)(x)\leq 0$ for all $x\in ((a,z^\ast]\cup[y^\ast,b))\setminus \mathcal{P}$, and
  \item[(iv)] $g'(x+)\leq g'(x-)$ for all $x\in((a,z^\ast]\cup[y^\ast,b))\cap \mathcal{P}$.
\end{itemize}
Then, $V(x)=V_{z^\ast,y^\ast}(x)$ and $\tau_{z^\ast,y^\ast}=\inf\{t\geq 0: X_t\not\in(z^\ast, y^\ast)\}$ is an optimal stopping time.
\end{lemma}
\begin{proof}
It is clear that under our assumptions $V_{z^\ast,y^\ast}(x)$ is nonnegative, continuous, and  dominates the exercise payoff $g(x)$ for all $x\in \mathcal{I}$.
Consider now the behavior of the mappings $(L_\psi V_{z^\ast,y^\ast})(x)$ and $(L_\varphi V_{z^\ast,y^\ast})(x)$. It is clear from \eqref{linearityGenerator} that
$(L_\psi V_{z^\ast,y^\ast})'(x)=(L_\varphi V_{z^\ast,y^\ast})'(x)=0$ for all $x\in(z^\ast,y^\ast)$ and $(L_\psi V_{z^\ast,y^\ast})'(x)=-\psi(x)(\mathcal{G}_rg)(x)m'(x) \geq 0,
(L_\varphi V_{z^\ast,y^\ast})'(x)=-\varphi(x)(\mathcal{G}_rg)(x)m'(x)\geq0$ for all $x\in((a,z^\ast)\cup (y^\ast,b))\setminus \mathcal{P}$. However, since
\begin{align*}
(L_u V_{z^\ast,y^\ast})(x-)-(L_u V_{z^\ast,y^\ast})(x+) = u(x)\frac{g'(x+)-g'(x-)}{S'(x)}\leq 0
\end{align*}
for all $x\in((a,z^\ast]\cup [y^\ast,b))\cap \mathcal{P}$ when $u=\psi$ or $u=\varphi$
we find that $(L_\psi V_{z^\ast,y^\ast})(x)$ and $(L_\varphi V_{z^\ast,y^\ast})(x)$ are nondecreasing on $\mathcal{I}$. Combining these observations with assumption (ii) then proves that
$(L_\psi V_{z^\ast,y^\ast})(x)\geq 0$ and $(L_\varphi V_{z^\ast,y^\ast})(x)\leq 0$ for all $x\in \mathcal{I}$.

Let $x_0\in(y^\ast,b)\setminus\mathcal{P}$ be a fixed reference point and define the ratio $h_{x_0}(x)=V_{z^\ast,y^\ast}(x)/V_{z^\ast,y^\ast}(x_0)=V_{z^\ast,y^\ast}(x)/g(x_0)$.
It is clear that our assumptions combined with identity \eqref{linearityGenerator} guarantee that
$$
\sigma_{x_0}^{h_{x_0}}((x,b])=
-\frac{\psi(x_0)}{Bg(x_0)}(L_\varphi g)(x+)
$$
is nonnegative and nonincreasing for all $x\geq x_0$ and $\sigma_{x_0}^{h_{x_0}}((x_0,b])=-\frac{\psi(x_0)}{Bg(x_0)}(L_\varphi g)(x_0)$. Analogously,
$$
\sigma_{x_0}^{h_{x_0}}([a,x))=
\frac{\varphi(x_0)}{Bg(x_0)}\left[(L_\psi g)(x-)\mathbbm{1}_{(a,z^\ast]\cup[y^\ast,x_0]}(x)+(L_\psi V_{z^\ast,y^\ast})(z^\ast-)\mathbbm{1}_{(z^\ast,y^\ast)}(x)\right]
$$
is nonnegative and nondecreasing for all $x\leq x_0$ and satisfies $\sigma_{x_0}^{h_{x_0}}([a,x_0))=\frac{\varphi(x_0)}{Bg(x_0)}(L_\psi g)(x_0)$. The identity $V(x)=V_{z^\ast,y^\ast}(x)$  and
optimality of the stopping time $\tau_{z^\ast,y^\ast}=\inf\{t\geq 0: X_t\not\in(z^\ast, y^\ast)\}$ results follow by utilizing analogous arguments with Lemma \ref{Martin1}.
\end{proof}

Lemma \ref{Martin2} states a set of sufficient conditions under which the considered stopping problem constitutes a two boundary problem where the
underlying diffusion is stopped as soon as it exits from the continuation region characterized by an open interval in the state space $\mathcal{I}$.
As in the case of Lemma \ref{Martin1} no differentiability at the stopping boundaries is required nor do we impose conditions on the monotonicity of
the generator $(\mathcal{G}_rg)(x)$ on $\mathcal{I}$. An interesting implication of the results of Lemma \ref{Martin2} is that at the optimal exercise
boundaries we have $V_{z^\ast,y^\ast}'(z^\ast-)\geq V_{z^\ast,y^\ast}'(z^\ast+)$ and $V_{z^\ast,y^\ast}'(y^\ast-)\geq V_{z^\ast,y^\ast}'(y^\ast+)$
where the inequalities may be strict in case the smooth fit principle is not satisfied. As we will observe later in this section in our explicit
numerical illustrations of our principal findings, it is precisely the non-differentiability of the value at the exercise threshold which may result
in situations where the function needed for the representation of the value as an expected supremum is discontinuous. Moreover, as in the single boundary setting, the potential non-monotonicity of the
generator on the stopping set may result in situations where the value of the optimal policy cannot be represented as an expected supremum.

\begin{remark}\label{sufficient2bound}
Assume that the following conditions are met:
\begin{itemize}
  \item[(i)] $(\mathcal{G}_rg)(x)\leq 0$ for all $x\in ((a,\tilde{x}_1)\cup(\tilde{x}_2,b))\setminus \mathcal{P}$, where $a<\tilde{x}_1<\tilde{x}_2<b$.
\item[(ii)] the mappings
$(L_\psi g)(x)$ and $(L_\varphi g)(x)$ are nondecreasing on $(a,\tilde{x}_1]\cup[\tilde{x}_2,b)$ and satisfy the limiting conditions $\lim_{x\downarrow a}(L_\psi g)(x)\geq 0$,
$\lim_{x\uparrow b}(L_\varphi g)(x)\leq 0$, $\lim_{x\uparrow b}(L_\psi g)(x)=\infty$, and $\lim_{x\downarrow a}(L_\varphi g)(x)=-\infty$.
\end{itemize}
Then, it can be shown by relying on the fixed point technique developed in \cite{Lempa10} and \cite{Ma12} that
there exists a candidate pair $z^\ast,y^\ast\in(a,\tilde{x}_1]\cup[\tilde{x}_2,b)$ maximizing $V_{z,y}(x)$ and resulting
in a $r$-excessive function $V_{z^\ast,y^\ast}(x)$. Especially, if $\mathcal{P}\subset(\tilde{x}_1,\tilde{x}_2)$, then
$z^\ast,y^\ast\in(a,\tilde{x}_1]\cup[\tilde{x}_2,b)$ constitutes the unique pair maximizing $V_{z,y}(x)$ and
$V(x)=V_{z^\ast,y^\ast}(x)$.
\end{remark}

In order to characterize the functions $f_1$ and $f_2$ and determine $J_{(z,y)}(x)$  explicitly, we first need to make some further assumptions.
\begin{assumption}\label{as 2}
We assume that either (a) $f_2(b)= f_1(a)$, or
(b) $f_2(b)>f_1(a)$,  $g(a)<\infty$, and $\lim_{x\uparrow b}g(x)/\psi(x)=0$.
\end{assumption}
It is at this point worthwhile to stress that a proof for the case "$f_2(b)<f_1(a)$, $g(b)<\infty$, and $\lim_{x\downarrow a}g(x)/\varphi(x)=0$" is completely analogous
with the proof in case (b) of Assumption \eqref{as 2}. Given these assumptions, define the state $\zeta:= f_2^{-1}(f_1(a+))$ and the functions $\alpha:[y,b)\mapsto(a,z]$
and $\beta:(a,z]\mapsto [y,\zeta)$ as (see Figure \ref{fig 2case})
\begin{align*}
\alpha(m)&:=f_1^{-1}(f_2(m))\\
\beta(i)&:=f_2^{-1}(f_1(i)).
\end{align*}
If these points do not exist, we interpret them by the generalized inverses:
\begin{align*}
f_2^{-1}(x)&=\inf\left\{m\in[y,b]\mid f_2(m)\geq x\right\}\\
f_1^{-1}(x)&=\sup\left\{i\in[a,z]\mid f_1(i)\geq x\right\}.
\end{align*}
Especially, we set $\alpha(m)=a$ for all $m\geq\zeta$ and notice that $\beta(i)\in[y,b)$ constitutes the point in the domain of $f_2$ for
which the indifference condition $f_1(i)=f_2(\beta(i))$ holds, whenever $f_1$ and $f_2$ are continuous at the points $i$ and $\beta(i)$, respectively.
Similarly, $\alpha(m)\in(a,z]$ constitutes a point in the domain of $f_1$ for which identity $f_1(\alpha(m))=f_2(m)$ holds, whenever $f_1$ and $f_2$
are continuous at $\alpha(m)$ and $m$, respectively. In order to ease the notations in the sequel, we shall denote these functions simply by $\alpha$
and $\beta$ omitting the variables $i$ and $m$ from the notation.
\begin{figure}[!ht]
\begin{center}
\includegraphics[width=0.5\textwidth]{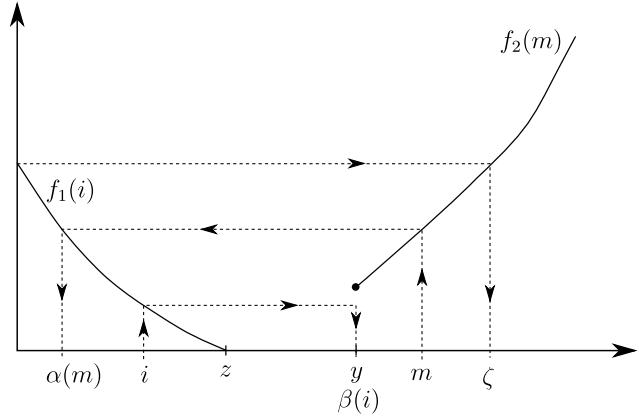}
\caption{\small Illustrating $f_1$, $f_2$, $\alpha$, $\beta$ and $\zeta$.}\label{fig 2case}
\end{center}
\end{figure}

\subsection{Calculating the expectation}
Utilizing the joint probability distribution \eqref{jointprob} described in Lemma \ref{probabilitydist} shows that
\begin{align}\label{eq px}
\begin{aligned}
\p_x(I_T\in di,M_T\leq \beta(i))&=\frac{-BS'(i)-\hat{\varphi}_\beta'(i)}{\hat{\varphi}_\beta^2(i)}\hat{\varphi}_\beta(x)di\\
\p_x(I_T\geq \alpha(m),M_T\in dm)&=\frac{-BS'(m)+\hat{\psi}_\alpha'(m)}{\hat{\psi}_\alpha^2(m)}\hat{\psi}_\alpha(x)dm\\
\end{aligned}
\end{align}
Given these densities, we notice that $J_{(z,y)}(x)$ can be rewritten as
\begin{align*}
\begin{aligned}
J_{(z,y)}(x)=\left\{
{\renewcommand{\arraystretch}{1.5}
\begin{array}{l@{\quad}l}
\D\int_a^xf_1(i) \p_x\left(I_T\in di,M_T < \beta(i)\right) +\int_{\beta(x)}^b f_2(m) \p_x\left(I_T>\alpha(m),M_T\in dm\right)&x\leq z\\
\D\int_a^zf_1(i) \p_x\left(I_T\in di,M_T < \beta(i)\right) +\int_y^b f_2(m) \p_x\left(I_T>\alpha(m),M_T\in dm\right)&x \in(z,y)\\
\D\int_a^{\alpha(x)}f_1(i) \p_x\left(I_T\in di,M_T < \beta(i)\right) +\int_{x}^b f_2(m) \p_x\left(I_T>\alpha(m),M_T\in dm\right)&x\geq y.
\end{array}}\right.
\end{aligned}
\end{align*}
Since our objective is to delineate circumstances under which $J_{(z,y)}(x)=V_{(z,y)}(x)$ holds especially for $x\in(z,y)$, we can first determine for which $f_1$ the equality
\begin{align*}
\lim_{x\mapsto z+}\frac{\partial}{\partial z}J_{(z,y)}(x)=\lim_{x\mapsto z+}\frac{\partial}{\partial z}V_{(z,y)}(x)
\end{align*}
holds. We can then make an {\em ansatz} that the solution of this identity constitutes the required function $f_1$. In a completely analogous fashion,
by differentiating $V_{(z,y)}$ with respect to $y$ and setting $x\mapsto y-$, we can make a second {\em ansatz}  that the solution of the resulting identity constitutes the required  $f_2$.
More precisely, we propose that the functions $f_1$ and $f_2$ should be of the form
\begin{align}\label{eq f12}
\begin{aligned}
f_1(i)&:=\frac{-g(\beta(i))BS'(i)+g'(i)\hat{\varphi}_{\beta(i)}(i)-g(i)\hat{\varphi}_{\beta(i)}'(i)}{-BS'(i)-\hat{\varphi}_{\beta(i)}'(i)}\\
f_2(m)&:=\frac{g(\alpha(m))BS'(m)+g'(m)\hat{\psi}_{\alpha(m)}(m)-g(m)\hat{\psi}_{\alpha(m)}'(m)}{BS'(m)-\hat{\psi}_{\alpha(m)}'(m)}.
\end{aligned}
\end{align}

\subsection{Verifying our ansatz}

Our objective is now to delineate circumstances under which our ansatz can be shown to be correct. To this end, at this point we assume that the problem specification is
such that  $f_1$ is non-increasing and $f_2$ is non-decreasing, otherwise the functions $\alpha$ and $\beta$ would not be unambiguously defined.
Later on, we shall state a set of sufficient conditions under which these monotonicity requirements indeed hold. In order to facilitate the explicit computation of the
functions $f_1$ and $f_2$, we assume in what follows that the boundaries $a$ and $b$ are natural for the underlying diffusion $X$.

Let us now compute $f_2(m)$ for $m\in[\zeta,b)\setminus \mathcal{P}$. We can rewrite $f_2$ as
\begin{align*}
f_2(m)=\frac{g(\alpha)BS'(m)\frac{1}{\varphi(\alpha)}+g'(m)\tilde{\psi}_{\alpha}(m)-
g(m)\tilde{\psi}_{\alpha}'(m)}{BS'(m)\frac{1}{\varphi(\alpha)}-\tilde{\psi}_{\alpha}'(m)},
\end{align*}
where $\tilde{\psi}_i(m)=\psi(m)-\frac{\psi(i)}{\varphi(i)}\varphi(m)$. Clearly, $\lim_{i\downarrow a}\tilde{\psi}_i(m)=\psi(m)$ and $\lim_{i\downarrow a}\tilde{\psi}_i'(m)=\psi'(m)$.
Moreover, since $a$ was assumed to be natural, and we interpreted $\alpha(m)=a$ for all $m\geq \zeta$, we get, for $m\in[\zeta,b)\setminus \mathcal{P}$, that
\begin{align}\label{eq f2zeta}
f_2(m)=g(m)-\frac{g'(m)}{\psi'(m)}\psi(m).
\end{align}
Similarly, applying  \eqref{eq px} shows that for all $m\geq\zeta$ it holds
\begin{align*}
\p_x(I_T\geq \alpha(m),M_T\in dm)&=\p_x(I_T\geq a,M_T\in dm)=\frac{\psi'(m)}{\psi^2(m)}\psi(x).
\end{align*}
We observe that these are, in fact, the very same functionals we got in Section \ref{sec 1sided} with the increasing one-sided case.
In order to verify our ansatz, let $x\in(z,y)$, and substitute $f_1$ and $f_2$ from \eqref{eq f12} and $\p_x$'s from \eqref{eq px} to $J_{(z,y)}(x)$. After reordering terms, we get
\begin{align*}
J_{(z,y)}(x)=&\varphi(x)\int_a^z\psi(\beta)\frac{-g(\beta) BS'(i)+g'(i)\hat{\varphi}_\beta(i)-g(i)\hat{\varphi}_\beta'(i)}{\hat{\varphi}_\beta^2(i)}di\\
+&\varphi(x)\int_y^\zeta \psi(\alpha)\frac{g(\alpha) BS'(m)+g'(m)\hat{\psi}_\alpha(m)-g(m)\hat{\psi}_\alpha'(m)}{\hat{\psi}_\alpha^2(m)}dm\\
-&\psi(x)\int_a^z\varphi(\beta)\frac{-g(\beta) BS'(i)+g'(i)\hat{\varphi}_\beta(i)-g(i)\hat{\varphi}_\beta'(i)}{\hat{\varphi}_\beta^2(i)}di\\
-
&\psi(x)\int_y^\zeta \varphi(\alpha)\frac{g(\alpha) BS'(m)+g'(m)\hat{\psi}_\alpha(m)-g(m)\hat{\psi}_\alpha'(m)}{\hat{\psi}_\alpha^2(m)}dm\\
+& \psi(x)\int_\zeta^b\frac{g(m)\psi'(m)-g'(m)\psi(m)}{\psi^2(m)}dm.
\end{align*}
Similar to one-sided case (Section \ref{sec 1sided}), we notice that the last integral $\int_\zeta^b()dm$ equals $g(\zeta)/\psi(\zeta)$. (Notice that it follows from our
assumptions that if $\zeta=b$, then $g(\zeta)/\psi(\zeta)=0$.)

Next let us make a change in variable in the integrals $\int_y^\zeta()dm$: Substitute $i:=\alpha(m)$ (or $m=\beta(i)$), so that $dm=\beta'(i)di$ and the boundaries change
as $y\mapsto \alpha(y)=:\hat{z}\leq z$ and $\zeta\mapsto a$. We notice that we can actually change the lower boundary as $y\mapsto z$, since for all $i\in(\hat{z},z)$ we
have $\beta'(i)=0$, showing that the integrand between $\hat{z}$ and $z$ equals zero. Doing this and reordering terms show that $J_{(z,y)}(x)$ can now be written as
\begin{align*}
J_{(z,y)}(x)&=\varphi(x)\int_a^z\frac{dA_1(i,\beta(i))}{di}di+\psi(x)\int_a^z\frac{dA_2(i,\beta(i))}{di}di+\psi(x)\frac{g(\zeta)}{\psi(\zeta)}\\
&=\varphi(x)\left(A_1(z,y)-A_1(a,\zeta)\right)+\psi(x)\left(A_2(z,y)-A_2(a,\zeta)\right)+\psi(x)\frac{g(\zeta)}{\psi(\zeta)}.
\end{align*}
Finally, since $a$ was assumed to be a natural boundary for $X$, we obtain that $A_1(a,\zeta)=0$ and $A_2(a,\zeta)=g(\zeta)/\psi(\zeta)$. Consequently,
$J_{(z,y)}(x)=A_1(z,y)\varphi(x)+A_2(z,y)\psi(x)=V_{(z,y)}(x)$ for $x\in(z,y)$ as claimed.

Verifying the validity of our ansatz for $x\notin(z,y)$ is entirely analogous. For $x\leq z$ we get
\begin{align*}
J_{(z,y)}(x)&=\varphi(x)\int_a^xdA_1(i,\beta(i))+\psi(x)\int_a^{x}dA_2(i,\beta(i))+\psi(x)\frac{g(\zeta)}{\psi(\zeta)}\\
&=\varphi(x)A_1(x,\beta(x))+\psi(x)A_2(x,\beta(x))=\frac{\hat{\varphi}_\beta(x)}{\hat{\varphi}_\beta(x)}g(x)+\frac{\hat{\psi}_x(x)}{\hat{\psi}_x(\beta)}g(\beta)=
g(x).
\end{align*}
For $x\in(y,\zeta)$ we get
\begin{align*}
J_{(z,y)}(x)&=\varphi(x)\int_a^{\alpha(x)}dA_1(i,\beta(i))+\psi(x)\int_a^{\alpha(x)}dA_2(i,\beta(i))+\psi(x)\frac{g(\zeta)}{\psi(\zeta)}\\
&=\frac{\hat{\varphi}_x(x)}{\hat{\varphi}_x(\alpha)}g(\alpha)+\frac{\hat{\psi}_\alpha(x)}{\hat{\psi}_\alpha(x)}g(x)=g(x).
\end{align*}
Finally, for $x\geq\zeta$ we get
\begin{align*}
J_{(z,y)}(x)=\psi(x)\int_x^b\frac{g(m)\psi'(m)-g'(m)\psi(m)}{\psi^2(m)}dm=g(x),
\end{align*}
where the equality follows from the derivation of the one-sided case \eqref{eq Jy}.
Let us now summarize the analysis done so far into the following theorem.

\begin{theorem}\label{esityslause}
Assume that  $z,y\in g^{-1}(\mathbb{R}_+)$ satisfy the condition $a<z<y<b$, that $a$ and $b$ are natural for $X$, and that Assumption \ref{as 2} holds.
Furthermore, assume that $f_1$ and $f_2$ are as in \eqref{eq f12}. Then, if $f_1$ is non-increasing and $f_2$ is non-decreasing, $J_{(z,y)}(x)=V_{z,y}(x)$.
Moreover, if inequality $f_1(z)\lor f_2(y)\geq 0$ is satisfied as well, then
$V_{z,y}(x)$ and $J_{(z,y)}(x)$ are $r$-excessive for $X$.
\end{theorem}
\begin{proof}
The validity of identity $J_{(z,y)}(x)=V_{z,y}(x)$ has been proven in the text. The alleged $r$-excessivity
of $J_{(z,y)}(x)$ and, consequently, $V_{z,y}(x)$ now follows from Lemma \ref{lemma J}.
\end{proof}

It is worth pointing out that we can replace assumption (B) of \eqref{as 2} with the condition "$f_2(b)<f_1(a)$, $g(b)<\infty$, and $\lim_{x\downarrow a}g(x)/\varphi(x)=0$"
and the analysis presented above still holds. In that case, we would need to define a point $\hat{\zeta}=f_1^{-1}(f_2(b))$, instead of $\zeta$. We also observe that
Theorem \ref{esityslause} does not require the continuity of the function $f$ (at the points $z$ and $y$)
since the monotonicity of $f_1$ and $f_2$ are sufficient for the equality $J_{(z,y)}=V_{z,y}(x)$.
As in the single boundary case, we again notice that these conditions do not guarantee that the value dominates the exercise  payoff.

\subsection{Conditions under which $f$ is as required}

In the statement of our problem, we assumed that $f(x)=f_1(x)\mathbbm{1}_{(a,z]}(x)+f_2(x)\mathbbm{1}_{[y,b)}(x)$, where $f_1(x)$ is non-increasing and $f_2(x)$
is non-decreasing. In this section we state a set of sufficient conditions under which these requirements are unambiguously fulfilled. Before stating our principal
characterization, we first make the following assumptions:
\begin{assumption}\label{Assumption}
Assume that the exercise payoff $g\in C^2(\mathcal{I})$ satisfies the conditions:
\begin{itemize}
  \item [(a)] There is a threshold $\hat{x}=\argmax\{(\mathcal{G}_rg)(x)\}\in \mathcal{I}$ such that $(\mathcal{G}_rg)(x)$ is nondecreasing on $(a,\hat{x})$, non-increasing on $(\hat{x},b)$, and
  $(\mathcal{G}_rg)(\hat{x})>0$,
  \item[(b)] $(\mathcal{G}_rg)(b-)\leq (\mathcal{G}_rg)(a+) < -\varepsilon,$ where $\varepsilon>0$.
\end{itemize}
\end{assumption}
It is worth noticing that assumptions (a) and (b) imply that there exists two states $a<x_0<x_1<b$ so that $x_0=\inf\{x\in (\mathcal{G}_rg)^{-1}(0)\}$ and
$x_1=\sup\{x\in (\mathcal{G}_rg)^{-1}(0)\}$, and $(\mathcal{G}_rg)^{-1}(0)\neq\emptyset$. Assumption (b) essentially guarantees that there exists a unique point $\zeta$
at which the increasing function $f_2(x)$ coincides with the one associated with the single
boundary setting characterized in \eqref{eq f2zeta}. We could naturally assume that $(\mathcal{G}_rg)(a+)\leq (\mathcal{G}_rg)(b-) < -\varepsilon,$ where
$\varepsilon>0$. In that case the point $\hat{\zeta}$ would be on the decreasing part $f_1(x)$. Since the analysis is completely analogous, we leave it for the
interested reader. Moreover, as was shown in \cite{Lempa10} and \cite{Ma12} our conditions are sufficient
for the existence of a unique extremal pair $z^\ast\in(a,x_0), y^\ast\in (x_1,b)$ s.t. $\tau_{z^\ast,y^\ast}=\inf\{t\geq0: X_t\notin (z^\ast,y^\ast)\}$
constitutes the optimal stopping time, $V_{z^\ast,y^\ast}(x)=V(x)$ constitutes the value of the optimal stopping problem, $C=(z^\ast,y^\ast)$ is the
continuation region, and $\Gamma=(a,z^\ast]\cup[y^\ast,b)$ is the stopping region.

The existence of a pair of monotonic and nonnegative functions $f_1$ and $f_2$ is proven in the following.
\begin{theorem}\label{thm 2suff}
Let Assumption \ref{Assumption} hold. Then, $f_1$ is non-increasing and $f_2$ is non-decreasing. Moreover, $f_1(z^\ast)=f_2(y^\ast)=0$ and
\begin{align*}
f_1(i)&=-\frac{\int_i^\beta\hat{\varphi}_\beta(t)(\mathcal{G}_rg)(t)m'(t)dt}{r\int_i^\beta\hat{\varphi}_\beta(t)m'(t)dt} = -\E_x[(\mathcal{G}_rg)(\tilde{X}_T)|\tilde{I}_T= i]\\
f_2(m)&=-\frac{\int_\alpha^m\hat{\psi}_\alpha(t)(\mathcal{G}_rg)(t)m'(t)dt}{r\int_\alpha^m\hat{\psi}_\alpha(t)m'(t)dt} = -\E_x[(\mathcal{G}_rg)(\hat{X}_T)|\hat{M}_T=m],
\end{align*}
where $\tilde{X}_t=\{X_t;t<\tau_\beta\}$, $\hat{X}_t=\{X_t;t<\tau_\alpha\}$, $\tilde{I}_t=\inf\{X_s;s\leq t\wedge\tau_\beta\}$, and $\hat{M}_t=\sup\{X_s;s\leq t\wedge\tau_\alpha\}$.
\end{theorem}
\begin{proof}
In order to establish the existence and monotonicity of the mappings $f_1,f_2$ consider first the functions
\begin{align*}
F_1^{y}(z)&=\frac{\frac{g'(z)}{S'(z)}\hat{\varphi}_{y}(z)-g(z)\frac{\hat{\varphi}_{y}'(z)}{S'(z)}-Bg(y)}{-B - \frac{\hat{\varphi}_{y}'(z)}{S'(z)}}\\
F_2^{z}(y)&=\frac{\frac{g'(y)}{S'(y)}\hat{\psi}_{z}(y)-g(y)\frac{\hat{\psi}_{z}'(y)}{S'(y)}+Bg(z)}{B-\frac{\hat{\psi}_{z}'(y)}{S'(y)}}
\end{align*}
derived in \eqref{eq f12}. Utilizing the identities \eqref{linearityGenerator} and \eqref{canonical} show that these mappings can be re-expressed in the simpler integral form
\begin{align}
F_1^{y}(z)&=\frac{(L_{\hat{\varphi}_{y}}g)(y)-(L_{\hat{\varphi}_{y}}g)(z)}{(L_{\hat{\varphi}_{y}}\mathbbm{1})(y)-(L_{\hat{\varphi}_{y}}\mathbbm{1})(z)}=
-\frac{\int_z^y(\mathcal{G}_rg)(t)\hat{\varphi}_{y}(t)m'(t)dt}{r\int_z^y\hat{\varphi}_{y}(t)m'(t)dt}\label{eq1}\\
F_2^{z}(y)&=\frac{(L_{\hat{\psi}_{z}}g)(y)-(L_{\hat{\psi}_{z}}g)(z)}{(L_{\hat{\psi}_{z}}\mathbbm{1})(y)-(L_{\hat{\psi}_{z}}\mathbbm{1})(z)} =
-\frac{\int_z^y(\mathcal{G}_rg)(t)\hat{\psi}_{z}(t)m'(t)dt}{r\int_z^y\hat{\psi}_{z}(t)m'(t)dt}.\label{eq2}
\end{align}
The alleged representation of the functions $f_1$ and $f_2$ follow directly from \eqref{eq1}, \eqref{eq2}, and Lemma \ref{jointprobabilitydist}
provided that the existence of a root of equation $F_1^{y}(z)=F_2^{z}(y)$ can be assured.
Utilizing the identities \eqref{eq1} and \eqref{eq2} show that the solutions have to satisfy identity $H(z,y)=0$, where $H:\mathcal{I}^2\mapsto \mathbb{R}$ is defined by
\begin{align}
H(z,y)= \int_z^y(\mathcal{G}_rg)(t)u_1(t)m'(t)dt,\label{nonlin}
\end{align}
and
$$
u_1(x) = \varphi(x)\left(\frac{\psi'(y)}{S'(y)}-\frac{\psi'(z)}{S'(z)}\right)-\psi(x)\left(\frac{\varphi'(y)}{S'(y)}-\frac{\varphi'(z)}{S'(z)}\right)
$$
is monotonically decreasing and $r$-harmonic and satisfies the boundary conditions $u_1(z)=\hat{\psi}_{z}'(y)/S'(y)-B > 0$, $u_1(y)= B+\hat{\varphi}_{y}'(z)/S'(z) < 0$, and $u_1'(z)S'(y)=u_1'(y)S'(z).$

We first notice that assumptions 4.6. (a) and (b) are sufficient for the existence of a unique pair $z^\ast\in (a,x_0), y^\ast\in (x_1,b)$ satisfying the optimality conditions
$(L_\psi g)(z^\ast)=(L_\psi g)(y^\ast)$ and $(L_\varphi g)(z^\ast)=(L_\varphi g)(y^\ast)$ implying that for any $r$-harmonic map
$u(x)=c_1\psi(x) + c_2\varphi(x), c_1,c_2\in \mathbb{R}$ we have
\begin{align*}
\int_{z^\ast}^{y^\ast}u(t)(\mathcal{G}_rg)(t)m'(t)dt=0.
\end{align*}
Thus, $H(z^\ast,y^\ast)=0$ showing that equation $H(z,y)=0$ has at least one solution such that $y_{z^\ast}=y^\ast\in (x_1,b)$. Moreover,
invoking \eqref{nc1}, \eqref{nc2}, and \eqref{eq f12} shows that the necessary
conditions for optimality of the pair $(z^\ast,y^\ast)$ coincide with the conditions $f_1(z^\ast)=0=f_2(y^\ast)$.

Given the results above, fix now $z\in (a,z^\ast)$ and consider the function $H(z,y)$.
Standard differentiation yields
that
\begin{align}
H_y(z,y)&=rm'(y)\int_{z}^{y}\hat{\varphi}_y(t)(\mathcal{G}_rg)(t)m'(t)dt-(\mathcal{G}_rg)(y)m'(y)\left(\frac{\hat{\varphi}_y'(y)}{S'(y)}-
\frac{\hat{\varphi}_y'(z)}{S'(z)}\right)\label{partial}\\
H_z(z,y) &= rm'(z)\int_{z}^{y}\hat{\psi}_z(t)(\mathcal{G}_rg)(t)m'(t)dt-(\mathcal{G}_rg)(z)m'(z)\left(\frac{\hat{\psi}_z'(y)}{S'(y)}-\frac{\hat{\psi}_z'(z)}{S'(z)}\right)\label{partialz}
\end{align}
demonstrating that $H(z,z)=H_y(z,z)=0$. Moreover, if $y\in(z, \hat{x}]$, then the monotonicity of the generator $(\mathcal{G}_rg)(x)$ on $(a,\hat{x})$
guarantees that $H_y(z,y)<0$ for all  $y\in(z, \hat{x}]$.
Hence, equation $H_y(z,y)=0$ does not have roots satisfying condition $z\neq y$ when $y\in(z,\hat{x})$. In a completely analogous fashion \eqref{partialz} shows that
$H(y,y)=H_z(y,y)=0$ and $H_z(z,y)<0$ for all $\hat{x}\leq z < y$. Hence, $H_z(z,y)=0$ does not have roots satisfying condition $z\neq y$ when $z\in(\hat{x}, y)$.
Given these observations, we notice that
the existence of a root $y_z\in (y^\ast,b)$ would be guaranteed provided that $\lim_{y\rightarrow b}H(z,y)>0$ for all $z\in (a,z^\ast)$.
To see that this is indeed the case, we first consider the limiting behavior of the function $\hat{H}:(a,\hat{x})\times(\hat{x},b)\mapsto\mathbb{R}$ defined as
$$
\hat{H}(z,y)=\frac{H(z,y)}{\left(\frac{\psi'(y)}{S'(y)}-\frac{\psi'(z)}{S'(z)}\right)\left(\frac{\varphi'(y)}{S'(y)}-\frac{\varphi'(z)}{S'(z)}\right)}.
$$
It is clear that
\begin{align*}
\hat{H}(z,y)=\frac{\int_z^y(\mathcal{G}_rg)(t)\varphi(t)m'(t)dt}{r\int_z^y\varphi(t)m'(t)dt}-
\frac{\int_z^y(\mathcal{G}_rg)(t)\psi(t)m'(t)dt}{r\int_z^y\psi(t)m'(t)dt}.
\end{align*}
Utilizing \eqref{limitingratio} now implies that for all $z\in \mathcal{I}$ we have
$$
\lim_{y\uparrow b}\frac{\int_z^y(\mathcal{G}_rg)(t)\psi(t)m'(t)dt}{r\int_z^y\psi(t)m'(t)dt}=\frac{1}{r}\lim_{y\uparrow b}(\mathcal{G}_rg)(y)=\frac{1}{r}(\mathcal{G}_rg)(b-)<0.
$$
Hence, for all $z\in \mathcal{I}$ it holds that
\begin{align*}
\lim_{y\rightarrow b-}\hat{H}(z,y)=\frac{\int_z^b\left((\mathcal{G}_rg)(t)-(\mathcal{G}_rg)(b-)\right)\varphi(t)m'(t)dt}{r\int_z^b\varphi(t)m'(t)dt}>0
\end{align*}
since $b=\argmin\{(\mathcal{G}_rg)(x)\}$.
The definition of $\hat{H}(z,y)$ now implies that
$\lim_{y\uparrow b}H(z,y)=\infty$ for all $z\in \mathcal{I}$. Thus, for all $z\in (a,z^\ast)$
equation $H(z,y)=0$ has a root $y_z\in(y^\ast,b)$. Moreover, implicit differentiation shows that for all $z\in (a,z^\ast)$ we have
$$
y_{z}'=-\frac{H_z(z,y)}{H_y(z,y)}=-\frac{m'(z)\int_z^y\left((\mathcal{G}_rg)(t)-
(\mathcal{G}_rg)(z)\right)\hat{\psi}_z(t)m'(t)dt}{m'(y)\int_z^y\left((\mathcal{G}_rg)(t)-(\mathcal{G}_rg)(y)\right)\hat{\varphi}_y(t)m'(t)dt} < 0
$$
proving the alleged monotonicity.
\end{proof}

\begin{remark}
Let $u(x)=c_1\psi(x)+c_2\varphi(x)\geq 0$, where $c_1,c_2\in \mathbb{R}$, and assume that $\xi_{u}$ is a random variable distributed on $(z,y)$ according to the
probability distribution $P_{u}$ with density
$$
p_{u}(t)=\frac{u(t)m'(t)}{\int_z^yu(t)m'(t)dt}.
$$
Then, our results demonstrate that the functions $f_1$ and $f_2$ can be determined from the stationary identity
\begin{align}
\mathbb{E}\left[(\mathcal{G}_rg)(\xi_{\varphi})\right] = \mathbb{E}\left[(\mathcal{G}_rg)(\xi_{\psi })\right].\label{ergodic1}
\end{align}
By utilizing standard ergodic limit results, identity \eqref{ergodic1} can alternatively be expressed as (cf. Section II.35 in \cite{BorSal02})
$$
\lim_{t\rightarrow \infty}\frac{\int_0^t (\mathcal{G}_rg)(X_s)\varphi(X_s)\mathbbm{1}_{(z,y)}(X_s)ds}{\int_0^t \varphi(X_s)\mathbbm{1}_{(z,y)}(X_s)ds}=
\lim_{t\rightarrow \infty}\frac{\int_0^t (\mathcal{G}_rg)(X_s)\psi(X_s)\mathbbm{1}_{(z,y)}(X_s)ds}{\int_0^t \psi(X_s)\mathbbm{1}_{(z,y)}(X_s)ds}.
$$
\end{remark}

Theorem \ref{thm 2suff} characterizes the functions $f_1$ and $f_2$ in a smooth setting. According to Theorem \ref{thm 2suff}, the functions $f_1$ and $f_2$ vanish at the optimal boundaries
$z^\ast$ and $y^\ast$, respectively. Moreover, according to Theorem \ref{thm 2suff}, the functions $f_1$ and $f_2$ can be expressed as conditional expectations of the generator
$(\mathcal{G}_rg)(x)$. The decreasing mapping $f_1(i)$ is associated with the diffusion $X$ killed at the state $\beta$ and its running infimum while
the increasing mapping $f_2(m)$ is associated with the diffusion $X$ killed at the state $\alpha$ and its running supremum.
Due to the interdependence of $f_1$ and $f_2$ it is not, however, clear beforehand whether the identities $f_1(z^\ast)=0$ and  $f_1(y^\ast)=0$ continue to hold in a less smooth framework.
As our subsequent examples indicate, there are cases under which these identities cease to hold as soon as the smooth pasting condition is not satisfied at one of the optimal exercise boundaries.

It is worth emphasizing that even though Theorem \ref{thm 2suff} assumes that the exercise payoff is smooth and that the boundaries of the state space of the
underlying diffusion are natural,
its results appear to be valid also under weaker regularity conditions and boundary classifications. More precisely, as is clear from the proof of Theorem \ref{thm 2suff} establishing
the existence and monotonicity
of the functions  $f_1$ and $f_2$ can essentially be reduced to the analysis of the identity
\begin{align}
\frac{(L_\varphi g)(y)-(L_\varphi g)(z)}{(L_\varphi \mathbbm{1})(y)-(L_\varphi \mathbbm{1})(z)}=
\frac{(L_\psi g)(y)-(L_\psi g)(z)}{(L_\psi \mathbbm{1})(y)-(L_\psi \mathbbm{1})(z)}.\label{keyeq}
\end{align}
Since the monotonicity and limiting behavior of the functionals $(L_\varphi g)(x)$ and $(L_\psi g)(x)$ is principally dictated by the behavior of the generator
$(\mathcal{G}_rg)(x)$ (when defined), one could, in principle, attempt to delineate more general circumstances under which the uniqueness of a monotone solution for
\eqref{keyeq} could be guaranteed. A natural extension which could be utilized to accomplish this task would be to rely on the weak formulation of Dynkin's theorem and,
essentially, focus on those rewards which admit the representation (see, for example, \cite{CrMo14},\cite{HeSt10}, and \cite{LaZe13})
$$
\E_x\left[e^{-r\tau}g(X_\tau)\mathbbm{1}_{\tau<\infty}\right]=g(x)+\E_x\left[\int_0^\tau e^{-rs}\tilde{g}(X_s)ds;\tau<\infty\right],
$$
where $\tilde{g}\in \mathcal{L}^1(\mathcal{I})$ coincides with the generator $(\mathcal{G}_rg)(x)$ whenever the payoff is sufficiently smooth.
It is clear from the proof of Theorem \ref{thm 2suff} that if the function $\tilde{g}$ satisfies parts (a) and (b) of Assumption \ref{Assumption}
with $\mathcal{G}_rg$ replaced by $\tilde{g}$ and $\tilde{g}$ is continuous outside a finite set of points in $\mathcal{I}$, then the identity
\begin{align*}
\frac{\int_z^y\psi(t)\tilde{g}(t)m'(t)dt}{r\int_z^y\psi(t)m'(t)dt}=
\frac{\int_z^y\varphi(t)\tilde{g}(t)m'(t)dt}{\int_z^y\varphi(t)m'(t)dt}
\end{align*}
generates a pair of functions $f_1,f_2$ satisfying our monotonicity requirements and characterizing the optimal exercise boundaries through the identities $z^\ast=\sup\{x\in \mathcal{I}: f_1(x)\geq 0\}$ and $y^\ast=\inf\{x\in \mathcal{I}: f_2(x)\geq 0\}$.

It is also clear that the second integral expression stated in Theorem \ref{thm 2suff} resembles the expression \eqref{esitys1} derived in the one-sided case.
This is naturally not surprising in light of the fact that the one-sided cases can be derived from the two-sided case as limiting cases. Our main observation on this is summarized in the following.
	
\begin{lemma}\label{lemma limita}
By setting $z\mapsto a$, we retrieve the situation of Theorem \ref{theorem inc}.
\end{lemma}

\begin{proof}
Since $f(x)=f_1(x)\mathbbm{1}_{(a,z]}(x)+f_2(x)\mathbbm{1}_{[y,b)}(x)$, we see that $\lim_{z\mapsto a} f(x)= f_2(x)\mathbbm{1}_{[y,b)}(x)$.
Moreover, now $\zeta=\beta(a)=\beta(z)=y$, and thus, just as we derived \eqref{eq f2zeta}, we get $f_2(m)=g(m)-\frac{g'(m)}{\psi'(m)}\psi(m),$ for $m\geq y$.
\end{proof}

\subsection{Connection with the optimal stopping signal}
As pointed out in the introduction, there is a large variety of settings under which the values of stochastic control problems can be represented in terms of the
expected value of the running supremum  (see, for example, \cite{BaBa}, \cite{BaElKa}, \cite{BaFo}, \cite{ChSaTa}, \cite{FoKn1}, and \cite{FoKn2}).
In what follows, our objective is to connect the developed approach to the optimal stopping signal approach developed in \cite{BaBa}.

Following \cite{BaBa}, consider now the functional
$$
\hat{F}(x;z,y):= \frac{\mathbb{E}_x\left[g(x)-e^{-r\tau_{z,y}}g(X_{\tau_{z,y}})\right]}{1-\mathbb{E}_x\left[e^{-r\tau_{z,y}}\right]},
$$
where $\tau_{z,y}=\inf\{t\geq0:X_t\not\in (z,y)\}$ denotes the first exit time from the open set $(z,y)\subset \mathcal{I}$. Applying our previous computations
yield that $\hat{F}$ can be re-expressed as
$$
\hat{F}(x;z,y)=\frac{\hat{\psi}_z(y)g(x)-g(z)\hat{\varphi}_y(x)-g(y)\hat{\psi}_z(x)}{\hat{\psi}_z(y)-\hat{\psi}_z(x)-\hat{\varphi}_y(x)}.
$$
Letting first $z\uparrow x$ and then $y\downarrow x$ in this expression yields (by applying L'Hospital's rule)
\begin{align*}
h_1(x,y):=\hat{F}(x;x-,y)&=\frac{g(x)\hat{\varphi}'_y(x)-g'(x)\hat{\varphi}_y(x)+BS'(x)g(y)}{\hat{\varphi}'_y(x)+BS'(x)}\\
h_2(x,z):=\hat{F}(x;z,x+)&=\frac{\hat{\psi}_z'(x)g(x)-g'(x)\hat{\psi}_z(x)-BS'(x)g(z)}{\hat{\psi}_z'(x)-BS'(x)}.
\end{align*}
Utilizing the proof of Theorem \ref{thm 2suff} shows that the functions $h_1, h_2$ can be re-expressed in the compact form
\begin{align*}
h_1(x,y)&=\frac{(L_{\hat{\varphi}}g)(x)-(L_{\hat{\varphi}_y}g)(y)}{(L_{\hat{\varphi_y}}\mathbbm{1})(x)-(L_{\hat{\varphi_y}}\mathbbm{1})(y)}\\
h_2(x,z)&=\frac{(L_{\hat{\psi_z}}g)(x)-(L_{\hat{\psi}_z}g)(z)}{(L_{\hat{\psi}_z}\mathbbm{1})(x)-(L_{\hat{\psi}_z}\mathbbm{1})(z)},
\end{align*}
proving that $h_1(x,y)=F_1^y(x)$ and $h_2(x,y)=F_2^z(x)$. Hence, we notice that the functions generating $f_1$ and $f_2$
coincide with the functions characterizing the behavior of  the functional $\hat{F}(x;z,y)$. Theorem 13 in \cite{BaBa}
tells us that the stopping set can in the present setting be represented in terms of the so called \emph{optimal stopping signal} $\gamma$ in the following way.
\begin{theorem}
The stopping set $\Gamma=\{x\in \mathcal{I}: g(x)=V(x)\}=\{x\in \mathcal{I}: \gamma(x)\geq 0\}$, where
\begin{align*}
\gamma(x)=\min_{y\neq x}
\begin{cases}
g(x)-g'(x)\frac{\psi(x)}{\psi'(x)},\quad & y=a\\
h_2(y,x),\quad &a<y<x\\
h_1(x,y),\quad &x<y<b.\\
g(x)-g'(x)\frac{\varphi(x)}{\varphi'(x)},\quad & y=b
\end{cases}
\end{align*}
\end{theorem}

If the smooth fit principle is met, then we know from Theorem \ref{thm 2suff} that the function $f(x)=f_1(x)\mathbbm{1}_{(a,z^*]}(x)+f_2(x)\mathbbm{1}_{[y^*,b)}(x)$
is positive on the same set as
$\gamma(x)$. In the next proposition we verify the intuitively clear fact that our $f(x)$ is indeed identical with $\gamma$ on the stopping set $\Gamma$.

\begin{proposition}
Let $f(x)=f_1(x)\mathbbm{1}_{(a,z^*]}(x)+f_2(x)\mathbbm{1}_{[y^*,b)}(x)$. Then $f(x)=\gamma(x)$ for $x\in \Gamma$.
\end{proposition}

\begin{proof}
Let us redefine $f$ on $(z^*,y^*)$ to be negative. In this way, we can write the stopping set $\Gamma=\{x\in \mathcal{I}: f(x)\geq0\}$.
Consider now the auxiliary parameterized stopping problem
\begin{align}\label{eq aux prob}
\sup_\tau\E_x\left[e^{-r\tau}\left(g(X_\tau)-k\right)\right],
\end{align}
where $k\geq0$ is an arbitrary positive constant and $g$ is as in the initial problem \eqref{eq prob}. We know by Theorem 13 from \cite{BaBa} that for the problem
\eqref{eq aux prob} the stopping set can be written as $\Gamma_k=\{x\in \mathcal{I}: \gamma(x)\geq k\}$. Thus, if we can also show that $\Gamma_k=\{x\in \mathcal{I}: f(x)\geq k\}$,
then we must necessarily have $f(x)=\gamma(x)$ as $k$ is arbitrary. In order to prove the desired result, let $f^k(x)=f^k_1(x)\mathbbm{1}_{(a,z^*]}(x)+f^k_2(x)\mathbbm{1}_{[y^*,b)}(x)$
be the function $f$ for the auxiliary problem \eqref{eq aux prob}. Using  representation \eqref{eq f12} now shows that $f^k_1(x)=f_1(x)-k$ and $f^k_2(x)=f_2(x)-k$. Hence, we have
$f_k(x)=f(x)-k$. Consequently, it follows that
\[\Gamma_k=\{x\in \mathcal{I}: f_k(x)\geq0\}=\{x\in \mathcal{I}: f(x)\geq k\}\]
and the claim follows.
\end{proof}

Unfortunately, neither function $\gamma$ nor $f$ can be expressed explicitly in a general setting despite the fact that they both constitute alternative representations of the same value.
The function $\gamma$ is too
complex due to the minimization operator. Although $f$ is more explicit than $\gamma$, it is nevertheless also too complex for explicit expressions due to the implicit connection
between $f_1$ and $f_2$ through $\alpha(m)$ and $\beta(i)$.  However, as our subsequent examples based on capped straddle options indicate, our approach applies even when the smooth pasting condition
is not met. In this respect the approach developed in our paper can generate the required representation in cases which do not appear
in the approach based on the stopping signal.

\subsection{Examples}
Since the functions $f_1$ and $f_2$ depend on each other, it is very hard to express these functions explicitly. Fortunately, the derived
integral representation is such that the functions can be solved numerically in an efficient way. In what follows we shall illustrate these functions and their intricacies in several
explicitly parameterized examples.

\subsubsection{Example 3: Minimum guaranteed payment option}
Set $\mathcal{I}=(0,\infty)$ and consider the optimal stopping problem
\begin{align}\label{guoshepp}
V^*(x)=\sup_{\tau}\E_x\left[e^{-r\tau}(X_\tau \vee c)\right],
\end{align}
where $c>0$ is an exogenously determined minimum guaranteed payment. As was shown in \cite{AlMatFin}, the assumed boundary behavior of the underlying diffusion process guarantees that
problem \eqref{guoshepp} has a two-sided solution with a value
\begin{align}\label{mathfinrep}
V^\ast(x)=V_{(z^{\ast},y^{\ast})}(x)=\begin{cases}
x &x\geq y^\ast\\
\frac{\hat{\varphi}_{y^\ast}(x)}{\hat{\varphi}_{y^\ast}(z^\ast)}c+\frac{\hat{\psi}_{z^\ast}(x)}{\hat{\psi}_{z^\ast}({y^\ast})}y^\ast &z^\ast<x< y^\ast\\
c &x\leq z^\ast
\end{cases}
\end{align}
where the thresholds $(z^\ast,y^\ast)$ constitutes the unique root of the first order optimality conditions
\begin{align*}
\frac{\psi'(y^\ast)}{S'(y^\ast)}y^\ast-\frac{\psi(y^\ast)}{S'(y^\ast)} &=\frac{\psi'(z^\ast)}{S'(z^\ast)}c\\
\frac{\varphi'(y^\ast)}{S'(y^\ast)}y^\ast-\frac{\varphi(y^\ast)}{S'(y^\ast)} &=\frac{\varphi'(z^\ast)}{S'(z^\ast)}c.
\end{align*}

\noindent{\bf Geometric Brownian motion:} Assume that $X_t$ constitutes a geometric Brownian motion characterized by the stochastic differential equation
\begin{align*}
dX_t=\mu X_tdt+\sigma X_tdW_t,
\end{align*}
where $\sigma>0$ and $\mu<r$. With these choices $\psi(x)=x^{\kappa_{+}}$, $\varphi(x)=x^{\kappa_{-}}$, where
\begin{align*}
\kappa_{\pm}=  \frac{1}{2}-\frac{\mu}{\sigma^2} \pm \sqrt{\left(\frac{1}{2}-\frac{\mu}{\sigma^2}\right)^2+\frac{2r}{\sigma^2}}
\end{align*}
are the solutions of the characteristic equation $\frac{1}{2}\sigma^2\kappa(\kappa-1)+\mu\kappa-r=0$.
Under these assumptions, problem \eqref{guoshepp} admits an explicit solution (cf. \cite{GuoShepp})
\begin{align*}
V^\ast(x)=V_{(z^{\ast},y^{\ast})}(x)=\begin{cases}
x &x\geq y^\ast\\
\left(\kappa_{+}\left(\frac{x}{z^\ast}\right)^{\kappa_-}-\kappa_{-}\left(\frac{x}{z^\ast}\right)^{\kappa_+}\right)\frac{c}{\kappa_{+}-\kappa_{-}} &z^\ast<x< y^\ast\\
 c &x\leq z^\ast
\end{cases}
\end{align*}
where
$$
z^\ast = \left(\frac{\kappa_{+}}{\kappa_{+}-1}\right)^{\frac{\kappa_{+}-1}{\kappa_{+}-\kappa_{-}}}\left(\frac{\kappa_{-}-1}{\kappa_{-}}\right)^{\frac{\kappa_{-}-1}{\kappa_{+}-\kappa_{-}}} c
$$
and
$$
y^\ast = \left(\frac{\kappa_{+}}{\kappa_{+}-1}\right)^{\frac{\kappa_{+}}{\kappa_{+}-\kappa_{-}}}\left(\frac{\kappa_{-}-1}{\kappa_{-}}\right)^{\frac{\kappa_{-}}{\kappa_{+}-\kappa_{-}}} c.
$$

Now the conditions of Theorem \ref{thm 2suff} are valid, so that we know that there exist a $f_1$ and $f_2$ such that $f_1$ is non-increasing and $f_2$ is
non-decreasing, $f_1(z^{\ast})=0=f_2(y^{\ast})$ and that $\E_x\left[\sup_{0\leq t\leq T}f(X_t)\right]=V_{(z^{\ast},y^{\ast})}(x)$ for
$f(x)=f_1(x)\mathbbm{1}_{(a,z]}(x)+f_2(x)\mathbbm{1}_{[y,b)}(x)$. It can be calculated that $\lim_{i\mapsto 0}f_1(i)=c$ and that
$\lim_{m\mapsto\infty}f_2(m)=\infty$, so that in this case $\zeta\neq b=\infty$. Unfortunately, the functions $f_1$ and $f_2$ cannot be expressed in analytically closed form. 
%\begin{figure}[!ht]
%\begin{center}
%\includegraphics[width=0.5\textwidth]{2case.eps}
%\caption{\small Illustrating $f_1$, $f_2$. The parameters are $\mu=0.02$, $\sigma=0.1$, $r=0.035$, $c=1$. With these choices $(z^{\ast},y^{\ast},\zeta)=(0.9156,1.1236,2.84713)$.}\label{fig ill2}
%\end{center}
%\end{figure}

\noindent{\bf Logistic Diffusion:}
Assume that $X_t$ constitutes a logistic diffusion process characterized by the stochastic differential equation
\begin{align*}
dX_t=\mu X_t(1-\gamma X_t)dt+\sigma X_tdW_t,
\end{align*}
where $\sigma>0,\gamma\geq0$ and $\mu>0$. In this case the fundamental solutions read as
\begin{align*}
\psi(x) &= x^{\kappa_+} M(\kappa_+,1+\kappa_+ -\kappa_{-}, 2\mu\gamma x/\sigma^2)\\
\varphi(x) &= x^{\kappa_{-}}
M(\kappa_{-},1-\kappa_+ +\kappa_{-}, 2\mu\gamma x/\sigma^2),
\end{align*}
where $M$ denotes the confluent hypergeometric
function. The functions $f_1$ and $f_2$ are now illustrated numerically in Figure \ref{fighyp1}.
\begin{figure}[!ht]
\begin{center}
\includegraphics[width=0.5\textwidth]{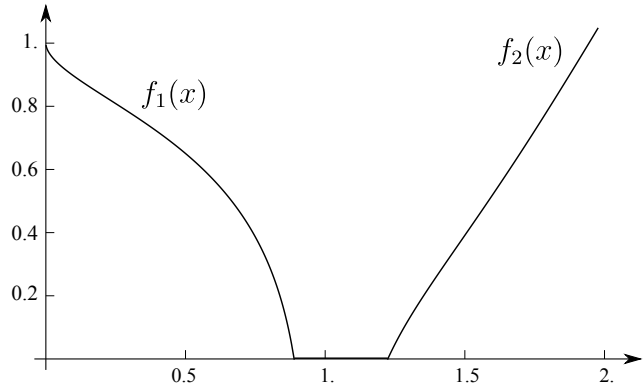}
\caption{\small Illustrating $f_1$, $f_2$ in logistic case. The parameters are $\mu=0.07$, $\sigma=0.1$, $\gamma=0.5$, $r=0.035$, $c=1$. With these choices
$(z^{\ast},y^{\ast},\zeta)=(0.8889,1.2242,1.9444)$.}\label{fighyp1}
\end{center}
\end{figure}

\subsubsection{Example 4: Capped straddle option}

We now assume that the underlying follows a GBM and focus on two straddle option variants. Namely, the symmetrically capped straddle with exercise payoff $g(x)=\min(|X-K|,C)$,
where $K>C>0$, and the asymmetrically capped straddle option with exercise payoff
$$g(x)=\min((K-x)^+,C_1) + \min((x-K)^+,C_2),$$
where $K>C_1>0, C_2>0$. It is worth noticing that the asymmetrically capped straddle is related to minimum guaranteed payoff option treated in the previous example, since
if $C_1<C_2$, then $g(x)\leq \max(C_1,\min((x-K)^+,C_2))$ and if $C_1>C_2$, then $g(x)\leq \max(C_2,\min((K-x)^+,C_1))$. In this way the value of the asymmetrically capped straddle
is dominated by the value of a minimum guaranteed payoff option.

It is now clear that the assumptions of our paper are met. Hence, the optimal exercise policy constitutes a two-boundary
stopping strategy. As in the capped call option case, the smooth fit condition may, however, be violated depending on the precise parametrization of the model. In the present example the functions
$f_1$ and $f_2$ are illustrated in Figure \ref{fig example 2} under diffusion parameter specifications resulting in $\psi(x)=x^2$ and $\varphi(x)=x^{-4}$.
\begin{figure}[!ht]
\begin{center}
\begin{subfigure}[b]{0.45\textwidth}
\begin{center}
\includegraphics[width=\textwidth]{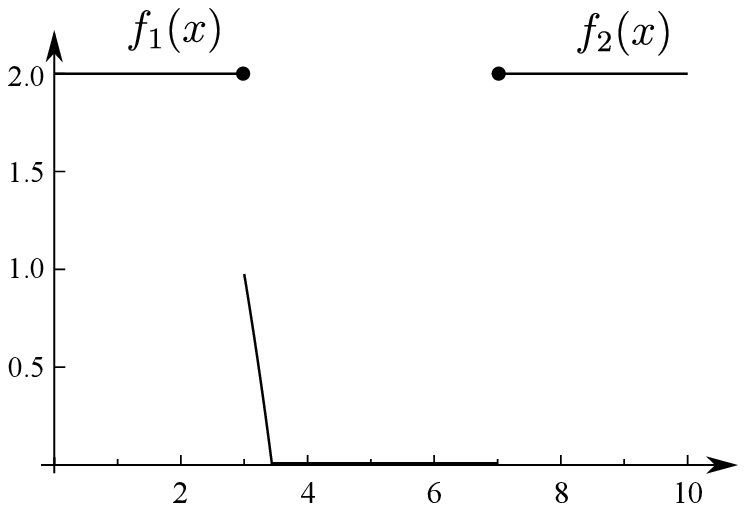}
\end{center}
\caption{Example 4: Capped straddle option with $g(x)=\min\{|x-5|,2\}$. Smooth fit at $z^*\approx 3.33$, corner solution at $y^*=7$. Both $f_1$ and $f_2$ are discontinuous.}\label{fig example 2a}
\end{subfigure}
\qquad
\begin{subfigure}[b]{0.45\textwidth}
\begin{center}
\includegraphics[width=\textwidth]{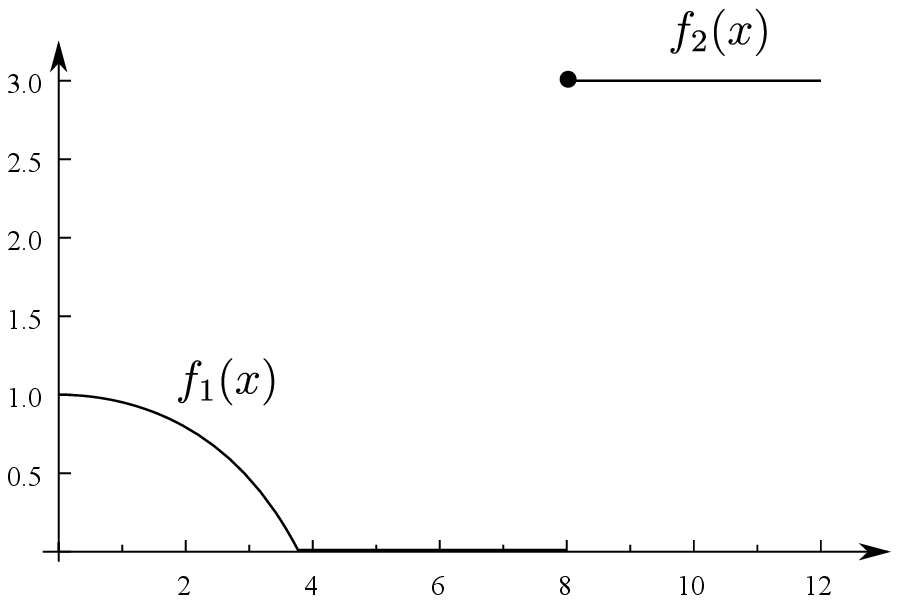}
\end{center}
\caption{Example 5: Asymmetrically capped straddle option with $C_1=1$, $K=5$, $C_2=3$. Smooth fit at $z^*\approx 3.78$ and corner solution at $y^*=9$. Now $f_1$ is continuous.}\label{fig example 2b}
\end{subfigure}
\end{center}
\caption{\small Numerical examples based on geometric Brownian motion.}\label{fig example 2}
\end{figure}
Under these specifications, we observe from Figure 5(A) that the functions $f_1$ and $f_2$ may be discontinuous. In the case of Figure 5(A),
the first discontinuity is based on the fact that the exercise payoff is not differentiable on the entire stopping region. The remaining discontinuity in
Figure 5(A) is based on the fact that the value does not satisfy the smooth fit principle at $y^\ast$. This observation illustrates the pronounced role of the interdependence between $f_1$ and $f_2$
and especially their sensitivity with respect to the potential nonsmoothness of the problem.

In both of these examples, $\zeta=y^*$, which enables us to write down the functions $f_1$ and $f_2$ explicitly. Especially, in the case of Figure \ref{fig example 2}(B), they are
\begin{align*}
\begin{aligned}
f_1(x)&=\frac{\frac{1}{2048}x^6-18x^2+256}{\frac{1}{2048}x^6-6x^2+256},\quad && i\in(0,z^*]\\
f_2(m)&=g(m)\equiv 3,\quad && m\in[y^*,\infty).
\end{aligned}
\end{align*}

\section{Conclusions}

We considered the representation of the value of an optimal stopping problem of a linear diffusion as the expected supremum of a function with known
regularity and monotonicity properties.
We developed an integral representation for the above mentioned function by first computing the joint probability distribution of the running supremum
and infimum of the underlying diffusion and then utilizing this distribution in determining the expected value explicitly in terms of the minimal excessive
mappings and the infinitesimal characteristics of the diffusion.

There are at least two  directions towards which our analysis could be potentially extended. First, given the close connection of optimal stopping with
singular stochastic control it would naturally be of interest to analyze if our representation would function in that setting as well. It is clear that
this should be doable at least in some circumstances, since typically the marginal value of a singular stochastic control problem can be interpreted as
a standard optimal stopping problem (see, for example, \cite{Ba,BaKa,BK,Kar1,KS1,KS2}). Such an extension would be very interesting especially from the
point of view of financial and economic applications, since a large class of control problems focusing on the rational management of a randomly fluctuating
stock can be viewed as singular stochastic control problems.  Second, impulse control and switching problems can in most cases be interpreted as sequential
stopping problems of the underlying process. Thus, extending our representation to that setting would be interesting too (for a recent approach to this
problem, see \cite{ChSa15}). However, given the potential
discreteness of the optimal policy in the impulse control policy setting seems to make the explicit determination of the integral representation a very
challenging problem which at the moment is outside the scope of our study.\\

\noindent {\bf Acknowledgements:} The authors are grateful to {\em Peter Bank} and {\em Paavo Salminen} for suggestions and helpful comments.\\

\bibliographystyle{amsplain}
\bibliography{Maks}

\end{document}